\documentclass[11pt,leqno]{amsart}
\usepackage{amssymb,mathrsfs,graphicx,color} 
\usepackage[all]{xy} 
\begin{document}
\theoremstyle{plain}
\newtheorem{thm}{Theorem}[section]
\newtheorem{prop}[thm]{Proposition}
\newtheorem{lemma}[thm]{Lemma}
\newtheorem{clry}[thm]{Corollary}
\newtheorem{definition}[thm]{Definition}
\newtheorem{hyp}{Assumption}
\newtheorem{claim}{Claim}
\newtheorem*{CSHyp}{Carleson-Sj\"olin Assumption}

\theoremstyle{definition}
\newtheorem{rem}[thm]{Remark}
\numberwithin{equation}{section}
\newcommand{\eps}{\varepsilon}
\newcommand{\e}{\mathrm{e}}
\renewcommand{\phi}{\varphi}
\renewcommand{\d}{\partial}
\newcommand{\dd}{\mathrm{d}}
\newcommand{\id}{\mathop{\rm Id} }
\newcommand{\re}{\mathop{\rm Re} }
\newcommand{\im}{\mathop{\rm Im}}
\newcommand{\R}{\mathbf{R}}
\newcommand{\T}{\mathbf{T}}
\renewcommand{\S}{\mathbf{S}}
\newcommand{\C}{\mathbf{C}}
\newcommand{\N}{\mathbf{N}} 
\newcommand{\Z}{\mathbf{Z}} 
\newcommand{\D}{\mathcal{C}^{\infty}_0} 
\newcommand{\supp}{\mathop{\rm supp}}
\newcommand{\grad}{\mathop{\rm grad}\nolimits}
\title[$L^p$ Resolvent estimates]{On $L^p$ resolvent estimates for  Laplace-Beltrami operators on compact manifolds}
\author[Dos Santos Ferreira, Kenig, Salo]{David Dos Santos Ferreira \and Carlos E.~Kenig  \and Mikko Salo }
\address{Universit\'e Paris 13, Cnrs, Umr 7539 Laga, 99 avenue Jean-Baptiste Cl\'ement, F-93430 Villetaneuse, France}
\email{ddsf@math.univ-paris13.fr}
\address{Department of Mathematics, University of Chicago, 5734 University Avenue, Chicago, IL 60637-1514, USA}
\email{cek@math.uchicago.edu}
\address{Department of Mathematics and Statistics, University of Helsinki and University of Jyv\"askyl\"a}
\email{mikko.salo@helsinki.fi}
\begin{abstract}
   In this article we prove $L^p$ estimates for resolvents of La-place-Beltrami operators on compact Riemannian manifolds, generalizing
   results of \cite{KRS} in the Euclidean case and \cite{Shen} for the torus. We follow \cite{Sog} and construct Hadamard's
   parametrix, then use classical boundedness results on integral operators with oscillatory kernels related to 
   the Carleson and Sj\"olin condition. Our initial motivation was to obtain $L^p$ Carleman estimates with limiting
   Carleman weights generalizing those of Jerison and Kenig \cite{JK}; we illustrate the pertinence of $L^p$ resolvent
   estimates by showing the relation with Carleman estimates. Such estimates are useful in the construction of complex
   geometrical optics solutions to the Schr\"odinger equation with unbounded potentials, an essential device for solving
   anisotropic inverse problems \cite{DKSa}.
\end{abstract}
\maketitle
\setcounter{tocdepth}{1} 
\tableofcontents
\hyphenation{para-metrix}
%
%
\begin{section}{Introduction}

This article aims at proving $L^p$ estimates on the resolvent of Laplace-Beltrami operators on compact Riemannian manifolds  
in the spirit of those obtained by Kenig, Ruiz and Sogge \cite{KRS} on the flat Euclidean space and of Shen \cite{Shen} on the flat torus. 
This work grew out of our recent interest in $L^p$ Carleman estimates with limiting Carleman weights 
generalizing Jerison and Kenig's estimates \cite{JK} to variable coefficients and used in the context of anisotropic inverse problems \cite{DKSa,DKSaU}.
It turned out that in \cite{DKSaU}, we were able to prove $L^p$ Carleman estimates in a direct way, without referring to resolvent
estimates. Meanwhile, in the process of obtaining those Carleman estimates, we were led to derive resolvent estimates.
Since these resolvent estimates might be helpful in different contexts, we decided to make them available.

Let $(M,g)$ be a compact Riemannian manifold of dimension $n \geq 3$ (without boundary). In local coordinates, the Laplace-Beltrami 
operator takes the form
\begin{align*}
    \Delta_g = \frac{1}{\sqrt{\det g}}  \frac{\d}{\d x_j} \bigg( \sqrt{\det g} \, g^{jk} \frac{\d}{\d x_k} \bigg)
\end{align*}
where $(g^{jk})_{1 \leq j,k \leq n}$ denotes the inverse matrix of $(g_{jk})_{1 \leq j,k \leq n}$ if $g=g_{jk} \,\dd x^j \otimes \dd x^k$.
As is classical, we are using Einstein's summation convention: repeated upper and lower indices are implicitly summed over all their possible values.
We will be using the notation
    $$ |\theta|_g = \sqrt{g_{jk} \theta^j \theta^k} $$
to denote the norm of a vector and $d_g$ to denote the geodesic distance in $(M,g)$. We denote $\sigma(-\Delta_g)$ the set of eigenvalues of the 
Laplace-Beltrami operator on $(M,g)$, and denote by
    $$ R(z) = (-\Delta_g+z)^{-1}, \quad z \in \C \setminus \sigma(-\Delta_g) $$
the resolvent operator. We are interested in proving estimates of the form
\begin{align*} 
      \|u\|_{L^{b}(M)} \leq C  \|(\Delta_{g}-z)u\|_{L^{a}(M)}
\end{align*}
for the couple of dual exponents $(a=2n/(n+2),b=2n/(n-2))$.
Note that by taking $u=1$, it is clear that such estimates can only hold away from $z=0$, and similarly by taking $u$ to be any eigenfunction of 
 $-\Delta_g$, away from the negative of any eigenvalue. This is in contrast to the non-compact flat case where Kenig, Ruiz and Sogge
 \cite{KRS} were able to obtain such estimates for all values of the complex parameter.  In fact, we will restrict our investigation to the 
 following set of values
 \begin{align*}
     \Xi_\delta &= \big\{z \in \C \setminus \R_- : \re \sqrt{z} \geq \delta  \big\} \\
     &=\big\{z \in \C \setminus \R_- : (\im z)^2 \geq 4\delta^2(\delta^2-\re z) \big\}.
\end{align*}
This set is the exterior of a parabola and can be visualized in figure \ref{fig:AdmSpec}.
The corresponding results were obtained by Shen \cite{Shen} on the flat torus. 

It seems that such estimates cannot be obtained as a direct consequence of the spectral cluster estimates obtained by Sogge \cite{Sog}. 
Indeed, if $\chi_m$ denotes the cluster of spectral projectors related to the eigenvalues of the Laplacian the square root of the negative 
of which lie in the interval $[m,m+1)$ then Sogge's estimates read
\begin{align*}
     \|\chi_m u\|_{L^{\frac{2n}{n-2}}(M)} &\leq C (1+m)^{\frac{1}{2}} \|u\|_{L^2(M)} \\
     \|\chi_m u\|_{L^{2}(M)} &\leq C (1+m)^{\frac{1}{2}} \|u\|_{L^{\frac{2n}{n+2}}(M)}. 
\end{align*}     
This would reduce the resolvent estimate to an $L^2$ estimate of the form
     $$ \|\chi_m u\|_{L^2(M)} \leq A(m,z) \|\chi_m (\Delta_g-z)u\|_{L^2(M)} $$ 
but summing the corresponding series $\sum_{m=0}^{\infty} (1+m)A(m,z)$ to patch the estimates requires a decay
of $A(m,z)$ which cannot be achieved. Nevertheless, we will use Sogge's estimates in the process of reducing 
Carleman estimates with limiting Carleman weights to resolvent estimates in the last section of this paper.
\begin{center}
\begin{figure}[h]
    \includegraphics[scale=0.5]{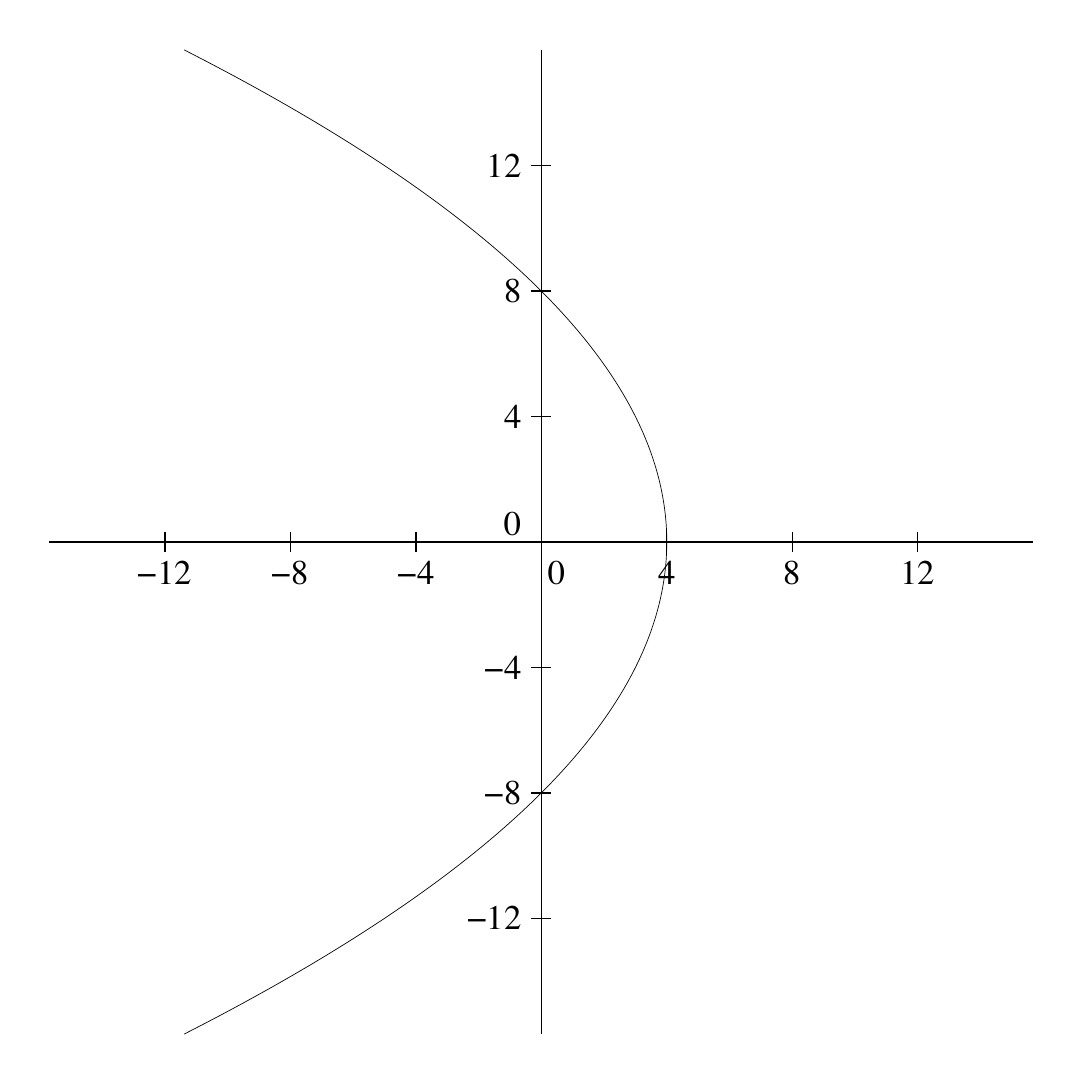}
     \caption{Allowed values of the spectral parameter $z$}
     \label{fig:AdmSpec}
\end{figure}
\end{center}
\begin{thm}
\label{Intro:MainThm}
     Let $(M,g)$ be a compact Riemannian manifold (without boundary) of dimension $n \geq 3$, and let $\delta \in (0,1)$ be a positive number.
     There exists a constant $C>0$ such that for all $u \in \mathcal{C}^{\infty}(M)$ and all $z\in \Xi_{\delta}$, the following resolvent estimate holds
     \begin{align} 
          \|u\|_{L^{\frac{2n}{n-2}}(M)} \leq C  \|(\Delta_{g}-z)u\|_{L^{\frac{2n}{n+2}}(M)}.
     \end{align}
\end{thm}
The proof of the theorem relies on the construction of the Hadamard parametrix.
We provide a fairly complete study of the $L^p$ boundedness properties of Hadamard's parametrix and the corresponding remainder term.
Some of the estimates might not be used in the proof of  Theorem \ref{Intro:MainThm} 
but may be of independent interest in other contexts.

\begin{rem}
      The question of whether the estimates hold for a larger set of values of the spectral parameter $z$, for instance
           $$ \widetilde{\Xi}_{\delta} = \{z \in \C \setminus \R_- : |\im z| \geq \delta \text{ when }\re z<0  \text{ and } 
                |z|>\delta \text{ when }\re z \geq 0\big\} $$
      remains open.
\end{rem}

As is common practice, we will write $A \lesssim B$ as shorthand for $A \leq CB$ where $C$
is a constant which depends on known fixed parameters (the dimension $n$, the indices $p,q$ of the Lebesgue classes involved, the Riemannian
manifold $(M,g)$, etc.),  but whose value is not crucial to the problem at hand. Similarly we use the notation $A \simeq B$ when $A \lesssim B$ 
and $A\gtrsim B$. We consider the principal branch of the square root on $\C \setminus \R_-$ which we denote by $\sqrt{z}$. 
We also choose the convention that $\arg z \in (-\pi,\pi)$ when $z \in \C \setminus \R_-$. \\

In the second section of this note, we review the different bounds known on oscillatory integral operators, in particular those obtained by the Carleson-Sj\"olin
theory. In the third section, we construct the Hadamard parametrix following \cite{Hor}. In the fourth section, we prove bounds on the Hadamard 
parametrix following the techniques of \cite{KRS,Sog} and derive the corresponding resolvent estimates. 
Finally, in the last section, we illustrate the pertinence of resolvent 
estimates by showing how they imply Carleman estimates with limiting Carleman weights in the spirit of Jerison and Kenig's estimates \cite{JK}.
Note that these Carleman estimates can be obtained by other means \cite{DKSa}.

\subsection*{Acknowledgements} Carlos Kenig was partially supported by NSF grant DMS-0968472. Mikko Salo is supported in part by
the Academy of Finland. David Dos Santos Ferreira and Mikko Salo would like to acknowledge the hospitality of the University of Chicago.
\end{section}
%
%
\begin{section}{Oscillatory integral operators}
This section is a short review of the classical boundedness results of operators with oscillatory kernels which will be needed
in this note. We consider operators of the form
    $$ T_{\lambda} u(x) = \int \e^{i \lambda \phi(x,y)} a(x,y) u(y) \, \dd y $$
where $a$ belongs to $\D(U \times V)$ with $U,V$ open sets in $\R^{n}$ and $\phi$ is a smooth function on $U \times V$. 
The product $T^*U \times T^*V$ of cotangent bundles of $U$ and $V$ endowed with the symplectic form $\sigma = \dd \xi \wedge \dd x + \dd \eta \wedge \dd y$
is a symplectic manifold of dimension $4n$. A fundamental object in the study of the operators under our scope is the Lagrangian submanifold of $T^*U  \times T^*V$
\begin{align*}
     \mathcal{C}_{\phi} = \big\{(x,\d_{x}\phi(x,y),y,\d_{y}\phi(x,y)) : (x,y) \in U \times V \big\}.
\end{align*} 
The first result relates to $L^2$ boundedness of such oscillatory integral operators. For that purpose, we consider the two projections
\begin{align*}
 \xymatrix @!0 @C=4pc @R=3pc {
    & \ar[ld]_{\pi_U} \mathcal{C}_{\phi} \subset T^*U\times T^*V  \ar[rd]^{\pi_V} &  \\
    T^*U & & T^*V.
  }
\end{align*}
If $\pi_{U}$ is a local diffeomorphism, i.e.
\begin{align}
\label{Osc:NonDegHyp}
     \det \frac{\d^2 \phi}{\d x \d y} \neq 0 \quad \text{ on } U \times V 
\end{align}
then so is $\pi_{V}$ and $\mathcal{C}_{\phi}$ is locally the graph of the canonical transformation $\varsigma = \pi_{U}^{-1} \circ \pi_{V}$.
The mother of $L^2$ estimates \cite{Hor} \cite[Theorem 2.1.1]{Sog3} corresponds to the case where $\mathcal{C}_{\phi}$ is locally a canonical graph.        
\begin{thm}
\label{Osc:L2Thm}
     Suppose that the non-degeneracy assumption \eqref{Osc:NonDegHyp} is satisfied, then for all compact $K \subset V$
     there exists a constant $C>0$ such that for all $\lambda \geq 1$ and all $u \in \D(K)$ 
     \begin{align*}
          \|T_{\lambda}u\|_{L^{2}} \leq C \lambda^{-n/2} \|u\|_{L^{2}}.
     \end{align*}
\end{thm}
This non-degeneracy assumption is too stringent for the applications in this note (we stated this result only for the sake of completeness),
in fact, we will use a weaker result and make the following non-degeneracy assumption \cite[Theorem 25.3.8]{Hor}  
\begin{align}
\label{Osc:Corank}
     \mathop{\rm corank} \sigma_{\mathcal{C}_{\phi}} = 2
\end{align}
on the lifted symplectic form
    $$ \sigma_{\mathcal{C}_{\phi}} = \pi_U^* \sigma. $$
Note that we could have used either the first or the second projection $\pi_{V}$ to lift the symplectic form; the condition is explicitly given by%
\footnote{This is also equivalent to the fact that $\dd\pi_{U} : T(\mathcal{C}_{\phi}) \to T(T^*U)$ has rank $2n-1$.}
\begin{align}
\label{Osc:CorankEquiv}
     \mathop{\rm corank} \frac{\d^2 \phi}{\d x \d y} = 1  \quad \text{ on } U \times V. 
\end{align}
The $L^2$ boundedness result that we will need reads as follows.
\begin{thm}
\label{Osc:DegL2Thm}
     Suppose that the non-degeneracy assumption \eqref{Osc:CorankEquiv} is satisfied, then for all compact $K \subset V$
     there exists a constant $C>0$ such that for all $\lambda \geq 1$ and all $u \in \D(K)$ 
     \begin{align*}
          \|T_{\lambda}u\|_{L^{2}} \leq C \lambda^{-(n-1)/2} \|u\|_{L^{2}}.
     \end{align*}
\end{thm}
By interpolation with the trivial $L^{\infty}-L^1$ bound, one gets the following $L^p-L^{p'}$ result.
\begin{clry}
\label{Osc:NonDegClry}
     Under the assumptions of Theorem \ref{Osc:DegL2Thm}, we have
     \begin{align*}
          \|T_{\lambda}u\|_{L^{p'}} \leq C_{p} \lambda^{-(n-1)/p'} \|u\|_{L^{p}}
     \end{align*}
     where $1 \leq p \leq 2$.
\end{clry}
The range of admissible exponents in this theorem is represented by the green line in figure \ref{fig:AdmExpBis}. 
    
The second type of results we need are those related to the $n \times n$ Carleson-Sj\"olin hypothesis \cite{CS}. 
Under this curvature assumption, the former $L^p-L^q$ boundedness results may be improved by a factor $\lambda^{-1/p}$.
We refer to \cite[Chapter IX]{Stein} and \cite[Chapter 2]{Sog3} for a more detailed exposition of such results.
Let us start by recalling this assumption:  suppose that \eqref{Osc:Corank}Ê\, holds, then 
\begin{align*}
      \Sigma_{x_{0}}^{U} &= \big\{\d_{x}\phi(x_{0},y) : y \in V \big\} \subset T_{x_{0}}^*U \\ \nonumber
      \Sigma_{y_{0}}^{V} &= \big\{\d_{y}\phi(x,y_{0}) : x  \in U \big\} \subset T_{y_{0}}^*V 
\end{align*}
are hypersurfaces in the respective cotangent spaces.
\begin{CSHyp}
         Both hypersurfaces $ \Sigma_{x_{0}}^{U}$ and $\Sigma_{y_{0}}^{V}$ have everywhere non-vanishing Gaussian curvature 
         for all $x_{0} \in U$ and $y_{0} \in V$.
\end{CSHyp}
Under the $n \times n$ Carleson-Sj\"olin hypothesis, we have the following boundedness result \cite[Corollary 2.2.3]{Sog3}.
\begin{thm}
\label{Osc:CSThm}
      Suppose that the non-degeneracy assumption \eqref{Osc:CorankEquiv} is satisfied and suppose that the $n \times n$ Carleson-Sj\"olin condition 
      holds. Then for any compact set $K \subset V$ there exists a constant $C>0$ such that 
      for all $\lambda \geq 1$ and all $u \in \D(K)$ 
     \begin{align}
          \|T_{\lambda}u\|_{L^q(\R^n)} \leq C \lambda^{-n/q} \|u\|_{L^p(\R^n)}
     \end{align}
     with $q = \frac{n+1}{n-1}p'$ and $1 \leq p \leq 2$.
\end{thm}
The range of admissible exponents in this theorem is represented by the red line in figure \ref{fig:AdmExpBis}.
\begin{center}
\begin{figure}
     \input{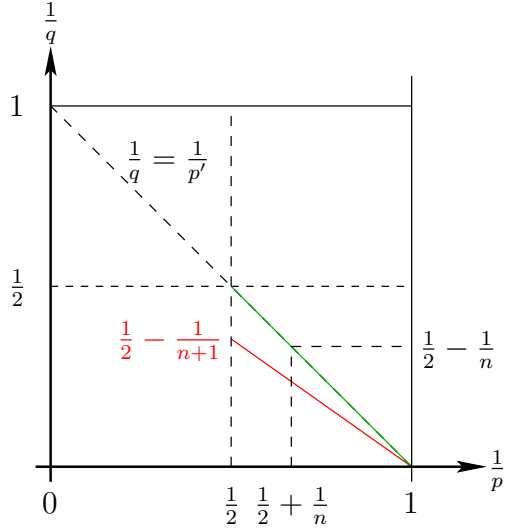}
     \caption{Admissible exponents in Theorem \ref{Osc:CSThm} and Corollary \ref{Osc:NonDegClry}}
     \label{fig:AdmExpBis}
\end{figure}
\end{center}
Let us finish by the following important remarks.
\begin{rem}
\label{Osc:Uniform}
     If the amplitude $a$ depends on parameters, the estimates in the theorems are uniform with respect to the parameters as long
     as the amplitude is supported in a fixed set $U \times V$ and is uniformly bounded  as well as all its derivatives. Similarly
     if the phase $\phi$ depends on parameters, the estimates remain uniform as long as all its derivatives are uniformly bounded and 
     either
     \begin{itemize}
           \item[--] there is a uniform lower bound on a minor of the determinant of the mixed Hessian in \eqref{Osc:CorankEquiv} in the case of Theorem 
           \ref{Osc:DegL2Thm} and Corollary \ref{Osc:NonDegClry},
           \item[--] or there is a uniform lower bound on the Gaussian curvature of the hypersurfaces in the case of Theorem \ref{Osc:CSThm}.
     \end{itemize}
\end{rem}
\begin{rem}
\label{Osc:RemSupport} 
      Although Theorems \ref{Osc:NonDegClry} and \ref{Osc:CSThm} are stated for oscillatory integral operators 
      $T_{\lambda}$ with amplitudes $a \in \D(U \times V)$ with constants that depend on the compact set $K = \supp a$, they are still valid 
      with constant depending on $R,\Lambda$ where we assume that
          $$ \big\{x-y : (x,y) \in \supp a \big\} \subset B(0,R) $$
      that $a$ and all its derivatives are bounded on $U \times V$ with bounds only depending on $\Lambda$
      and with phase $\phi$ satisfying non-degeneracy or curvature assumptions of the type discussed in Remark \ref{Osc:Uniform}, 
      uniformly on $U \times V$, with constant depending only on $1/\Lambda$.
      
      By the compact support assumption on the amplitude $a$ in $U\times V$, we have a natural extension of $T_{\lambda}$ to an operator
      acting on functions in $\R^n$, which gives a function on $\R^n$. We proceed as in \cite[pages 392--393]{Stein}.     
      Firstly, one observes that it suffices to prove an estimate of the form
      \begin{align}
      \label{Osc:LocEst}
           \|T_{\lambda}u\|_{L^q(B(x_0,1))} \leq C \lambda^{-s} \|u\|_{L^p(B(x_0,r))} 
      \end{align}
       with a constant $C$ uniform with respect to $x_0 \in \R^n$  and some radius $r$ depending only on $R$.
       In fact, if \eqref{Osc:LocEst} holds, raising both sides of \eqref{Osc:LocEst} to the $q$-th power $q$ and integrating with respect to $x_0 \in \R^n$ gives
       \begin{align*}
           |B(0,1)| \, \|T_{\lambda}u\|^q_{L^q}  \leq C^q \lambda^{-sq} \int \bigg( \int_{|x-x_0| <r} |u(x)|^p \, \dd x\bigg)^{\frac{q}{p}} \, \dd x_0
       \end{align*}
        and by Minkowski integral inequality ($q \geq p$)
        \begin{align*}
            |B(0,1)| \, \|T_{\lambda}u\|^q_{L^q}  \leq C^q \lambda^{-sq} \bigg(\int |u(x)|^p |B(0,r)|^{\frac{p}{q}} \, \dd x\bigg)^{\frac{q}{p}}
        \end{align*}
         so that
        \begin{align*}
             \|T_{\lambda}u\|_{L^q} \leq C \bigg(\frac{|B(0,r)|}{|B(0,1)|}\bigg)^{\frac{1}{q}} \lambda^{-s} \|u\|_{L^p}.
        \end{align*}
        
         Secondly, the estimate \eqref{Osc:LocEst} follow from the $L^p-L^q$  Theorems \ref{Osc:NonDegClry} and \ref{Osc:CSThm} 
         discussed in Remark \ref{Osc:Uniform} because the operator $T_{\lambda}$ does not move too much the support: let $\chi \in \D(B(0,2))$, 
         resp. $\psi \in \D(B(0,2r)$ be cutoff functions which equal one on $B(0,1)$, resp. $B(0,r)$. Then one has $\chi(x-x_0)a(x,y)=\chi(x-x_0)a(x,y)\psi(y-y_0)$
         for $r$ large depending on $R$ and 
        \begin{align*}
              \|T_{\lambda}u\|_{L^q(B(x_0,1))} \leq  \|\chi(\cdot -x_0) T_{\lambda}u\|_{L^q} = \|\chi(\cdot-x_0) T_{\lambda} \psi(\cdot-x_0)u\|_{L^q}.
        \end{align*} 
        Finally, the estimate \eqref{Osc:LocEst} follows from the $L^p-L^q$ boundedness results on the oscillatory integral operator
        \begin{align*}
             \int \chi(x) a(x+x_0,y+y_0) \psi(y) \e^{i \lambda \phi(x+x_0,y+y_0)} u(y) \, \dd y
        \end{align*} 
        which are uniform in $x_0$ by our assumptions on $a,\phi$.
\end{rem}
\end{section}
%
%
\begin{section}{Hadamard's parametrix}
\label{Sec:Had}

We follow \cite{Sog}, and use Hadamard's parametrix \cite[Section 17.4]{Hor} to prove the resolvent estimates. We first introduce the 
following sequence of functions
\begin{align*}
     F_{\nu}(|x|,z) =  \nu !  \, (2\pi)^{-n} \int_{\R^n} \frac{\e^{i x \cdot \xi}} {(|\xi|^2+z)^{1+\nu}} \, \dd \xi, \quad \nu \in \N
\end{align*}
including a fundamental solution $F_0$ of $(-\Delta+z)$ where $\Delta$  is the flat Laplacian on $\R^n$.
We used $|\cdot |$ for the Euclidean norm (recall that $|\cdot|_g$ denotes the norm induced by the Riemannian metric). 
These are radial functions and in polar coordinates, the equation on $F_{\nu}$ reads
\begin{align*}
     (-\Delta+z)F_{\nu} = -\d_{r}^2 F_{\nu} - \frac{n-1}{r} \d_{r}F_{\nu} +z F_{\nu} =
     \begin{cases}
          \delta_0 & \text{ when } \nu=0 \\ \nu F_{\nu-1} & \text{ when } \nu >0
     \end{cases}.
\end{align*}
Besides, as can be seen from their Fourier transform, these functions satisfy
\begin{align}
\label{Had:Recur}
     \d_rF_{\nu} = -\frac{r}{2} F_{\nu-1}, \quad \nu >0.
\end{align}
  
\subsection{Construction of the parametrix}

Let $x_{0} \in M$ and $U$ be a neighbourhood of $x_{0}$ in $M$. Since the construction is local, we can think of $U$ as a ball in $\R^n$, 
and work in coordinates. If $U$ is small enough, the distance function $r=d_g(x,y)$ is smooth on 
the product $U \times U$ with the diagonal removed, and furthermore one can take polar normal coordinates for $x$ in $U$ with center at $y$.
In polar coordinates 
   $$ x = \exp_{y}(r \theta), \quad \theta \in T_{y}M $$
the volume form takes the form
   $$ \dd V_g = r^{n-1} J(r,\theta) \, \dd r \wedge \dd \theta $$
where $\dd \theta$ denotes the canonical measure on the unit sphere $S_{y}M$ of the tangent space $(T_{y}M,g(y))$.
The Laplacian of the distance function $r=d_g(x,y)$ is given by%
\footnote{We refer the reader to \cite[section 4.B]{GHL} for those computations.}
   $$ \Delta_g r = \frac{n-1}{r} + \frac{\d_{r}J(r,\theta)}{J(r,\theta)} $$
and this implies for a radial function
\begin{align*}
     \Delta_g f(r) &= f''(r)|\dd r|_g^2 + f'(r) \Delta_g r  \\ &= f''(r) + \frac{n-1}{r}f'(r) + \frac{\d_{r}J}{J} f'(r). 
\end{align*}
This identity holds in the classical sense when $x \neq y$ and in the distribution sense on $U$.
If we use those computations on the function $\alpha_{0}F_{0}$, we get 
\begin{align*}
      (-\Delta_g+z)(\alpha_{0} F_{0}) = \alpha_0 \delta_0 - 2\bigg(\d_r\alpha_{0} +\frac{\d_r J}{2J} \alpha_{0} \bigg)\d_r F_{0}-(\Delta_g \alpha_{0})F_{0}
\end{align*}
 and on the function $\alpha_{\nu}F_{\nu}, \, \nu>0$ using \eqref{Had:Recur}
\begin{align*}
      (-\Delta_g+z)(\alpha_{\nu} F_{\nu}) = \bigg(r\d_r\alpha_{\nu} + \bigg(\nu+r \frac{\d_r J}{2J}\bigg)\alpha_{\nu}\bigg) F_{\nu-1}
      -(\Delta_g \alpha_{\nu})F_{\nu}
\end{align*}     
in $\mathcal{D}'(U)$. Thus, if we set     
     $$ F(x,y,z) = \sum_{\nu=0}^N \alpha_{\nu}(x,y) F_{\nu}\big(d_g(x,y),z\big), $$ 
then we have
\begin{multline*}
     (-\Delta_g+z) F = \alpha_0 \delta(r)   - 2\bigg( \d_r \alpha_0 + \frac{\d_{r}J}{2J} \alpha_{0}\bigg) \d_rF_0 \\
     +\sum_{\nu=1}^{N}\bigg(r\d_r \alpha_{\nu}+\bigg(\nu+r\frac{\d_rJ}{2J}\bigg)\alpha_{\nu}-\Delta_g \alpha_{\nu-1}\bigg)F_{\nu-1}
     -(\Delta_g \alpha_N) F_N
\end{multline*}
in $\mathcal{D}'(U)$, where the derivatives were taken with respect to $x$. Here $\delta(r)$ stands for the pullback of the Dirac mass $\delta_0$
on $\R$ by the map $x \mapsto d_g(x,y)$, this can be explicitly computed 
\begin{align*}
    \delta(r) =\frac{1}{\sqrt{\det g(y)}} \, \delta_{y}(x).
\end{align*}

If we choose the coefficients $\alpha_{\nu}$ to be solutions of the equations
\begin{align}
      \d_r \alpha_0 + \frac{\d_{r}J}{2J} \alpha_{0} &= 0 \label{Had:TranspZero}\\
      r\d_r \alpha_{\nu}+\bigg(\nu+r\frac{\d_rJ}{2J}\bigg)\alpha_{\nu}&=\Delta_g \alpha_{\nu-1}, \quad 1 \leq \nu \leq N \label{Had:TranspNu}
\end{align}
then the function $F$ satisfies:
\begin{align}
\label{Had:deltaEq}
     (-\Delta_g+z) F = \alpha_0 \delta(r)  -(\Delta_g \alpha_N) F_N.
\end{align}
The equation \eqref{Had:TranspZero} can be explicitly solved by taking
\begin{align*}
     \alpha_0 = \det g(y)^{-1/4} \frac{1}{\sqrt{J(r,\theta)}} = \bigg(\frac{\det g(y)}{\det (\exp_y^*g)(\exp_y^{-1}x)}\bigg)^{-\frac{1}{4}}.
\end{align*}
Note that the computations that we just made are not coordinate invariant, neither is the quantity $\det g$ (contrary to the volume form) 
so the choice of $\alpha_0$ may be understood in the following manner: first pick normal coordinates in $x$ depending on $y$ 
     $$ (x,y) \to (\exp^{-1}_y(x),y) $$
then $\alpha_0$ depends on the ratio of the determinant of the metric expressed in the coordinates in $y$ and of the determinant of the metric 
expressed in the normal coordinates (depending on $y$!) in $x$ that we have just picked.  Having given those precisions, it is clear that the function 
$\alpha_0$ is smooth and satisfies $\alpha_0(y,y)=1$.

Setting $\widetilde{\alpha}_{\nu}=\alpha_0^{-1} \alpha_{\nu}$ and $\widetilde{\beta}_{\nu-1} =\alpha_0^{-1} \Delta_g \alpha_{\nu-1}$ 
reduces the equation \eqref{Had:TranspNu} to
\begin{align*}
      \underbrace{r\d_r \widetilde{\alpha}_{\nu}+\nu\widetilde{\alpha}_{\nu}}_{=r^{1-\nu}\d_r(r^{\nu}\widetilde{\alpha}_{\nu})}
      &= \widetilde{\beta}_{\nu-1}, \quad 1 \leq \nu \leq N.
\end{align*}
In polar coordinates, this is explicitly solved by
\begin{align*}
     \widetilde{\alpha}_{\nu} = r^{-\nu} \int_0^r t^{\nu-1} \widetilde{\beta}_{\nu-1} (\exp_y t \theta) \, \dd t = 
     \int_0^1 t^{\nu-1} \widetilde{\beta}_{\nu-1} (\exp_y rt \theta) \, \dd t, \quad \nu \geq 1.
\end{align*}
By induction, one can solve the equations on $\alpha_{\nu}$ by taking
\begin{align*}
     \alpha_{\nu} = \alpha_0 \int_0^1 t^{\nu-1} \big(\alpha_0^{-1}\Delta_g \alpha_{\nu-1}\big) (\gamma_{x,y}(t)) \, \dd t.
\end{align*}
where $\gamma_{x,y}$ is the distance minimizing geodesic from $y$ to $x$ (this is well defined if the neighbourhood $U$ of $x_0$ is small enough). 
These functions are smooth on $U$.

Let $\chi \in \D(U \times U)$ be a symmetric function which equals one near the diagonal, we get using our computation \eqref{Had:deltaEq}
\begin{multline}
     (-\Delta_g+z) \Big[\chi(\cdot \,,y) F(\cdot \,,y,z)\Big] = \frac{\chi (y,y)}{\sqrt{\det g(y)}} \, \delta_{y}  \\ +  
     [\chi(\cdot,y),\Delta_g] F(\cdot \,,y,z) + \underbrace{\chi(\cdot,y)(\Delta_g \alpha_N)F_N\big(d_g(\cdot \,,y),z\big)}_{=H_N(x,y,z)}.
\end{multline}
Note that $[\chi(\cdot,y),\Delta_g]=2\grad_g \chi(\cdot,y)$ is a first order differential operator whose coefficients 
are supported in the support of $\chi$ and away from the diagonal $x=y$. We are ready to consider Hadamard's parametrix
\begin{align*}
     T_{\rm Had}(z)u = \int_{M} \chi(x,y) F(x,y,z) u(y)  \, \dd V_g(y)
\end{align*}
which satisfies
\begin{align}
\label{Had:ParmaIdentity}
    (-\Delta_g+z) T_{\rm Had}(z) u = \chi(x,x) u + S(z)u
\end{align}
with $S(z)= 2\grad_g \chi(\cdot,y) \circ T_{\rm Had}(z)+S_2(z)=S_1(z)+S_2(z)$ and
\begin{align*}
     S_2(z) = \int_{M} H_N\big(x,y,z) u(y)  \, \dd V_g(y).
\end{align*}
The integer $N$ will be chosen, in the next paragraph, large enough so that $S_2(z)$ is smoothing.
Given the fact that the function 
      $$ F(y,x,z) = \sum_{\nu=0}^N \alpha_{\nu}(y,x) F_{\nu}\big(d_g(x,y)\big) $$
behaves in a similar fashion to $F(x,y,z)$ (the only lack of symmetry comes from the coefficients $\alpha_{\nu}$ whose only relevant property
is their smoothness), the transpose ${}^tT_{\rm Had}(z)$ shares the same boundedness properties than 
$T_{\rm Had}(z)$. It is therefore not misleading to think of the Hadamard parametrix as a symmetric operator in terms
of its boundedness properties.
     
\subsection{Bessel functions}   
   
In this paragraph we describe the behaviour of the functions $F_{\nu}$ (see \cite[page 338--339]{KRS}, \cite[Lemma 4.3]{Sog}).
They can be explicitly computed in terms of Bessel functions
\begin{align}
\label{Had:Bessel}
     F_{\nu}(r,z) &= c_{\nu} r^{-\frac{n}{2}+\nu+1} z^{\frac{n}{4}-\frac{\nu+1}{2}} K_{n/2-1-\nu}\big(\sqrt{z}r\big) 
\end{align}
where $ K_{n/2-1-\nu}$ are Bessel potentials
\begin{align*}
       K_{m}(w) = \int_0^{\infty} \e^{-w \cosh t} \cosh(m t) \, \dd t, \quad \re w >0. 
\end{align*}
The qualitative properties of $K_m$ are as follows
\begin{align*}
     |K_m(w)| &\leq C_m |w|^{-m} \quad &&\text{when }  |w| \leq 1 \text{ and } \re w>0, \\
     K_m(w) &= a_m(w) w^{-\frac{1}{2}} \e^{-w} &&\text{when } |w| \geq 1 \text{ and }\re w>0, 
\end{align*}
where the functions $a_{m}$ have uniform estimates of the form
\begin{align}
      \bigg|\frac{\d^\alpha}{\d r^\alpha}a_{m}\Big(r\frac{w}{|w|}\Big)\bigg| \leq C_{\alpha, m} r^{-\alpha}, \quad  \alpha \in \N
\end{align}
when $r \geq 1$. We sum up the well known properties of the function $F_{\nu}$ obtained from the explicit formula \eqref{Had:Bessel} in the following lemma.
\begin{lemma}
\label{Had:BesselDescr}
     There is a constant $C>0$ such that
     \begin{align*}
          |F_{\nu}(r,z)| \leq C r^{-n+2+2\nu} \quad \text{ when } r \leq |z|^{-\frac{1}{2}}.
     \end{align*}
     Furthermore, one has
     \begin{align*}
          F_{\nu}(r,z) &= |z|^{\frac{n-1}{4}-\frac{\nu+1}{2}} \e^{-\sqrt{z} r} r^{-\frac{n-1}{2}+\nu} a_{\nu}(r,z)
     \end{align*}
     where the functions $a_{\nu}$ have a symbolic behaviour
     \begin{align*}
           \bigg|\frac{\d^\alpha a_{\nu}}{\d r^\alpha}(r,z)\bigg| \leq C_{\alpha,\nu} r^{-\alpha}, \quad \alpha \in \N
     \end{align*}
     when $r \geq |z|^{-\frac{1}{2}}$. The constants are all uniform with respect to $z \in \C$, $|z| \geq c$.
\end{lemma} 
In particular, if $N>(n-1)/2$ then by  Lemma \ref{Had:BesselDescr} we have 
\begin{align}
\label{Had:RemGood}
      |F_N(r,z)| \lesssim |z|^{-1/2}
\end{align}
when $|z| \geq \delta$. We also recall that 
      $$ F(x,y,z) = \sum_{\nu=0}^N \alpha_{\nu}(x,y) F_{\nu}\big(d_g(x,y),z\big) $$
with smooth coefficients $\alpha_{\nu}$. Thus from Lemma \ref{Had:BesselDescr} we deduce bounds and asymptotics 
on the function $F$, which we synthesize in the following lemma.
\begin{lemma}
\label{Had:FDescr}
     There is a constant $C>0$ such that
     \begin{align}
       \label{Had:Asymp1}
          |F(x,y,z)| \leq C d_g(x,y)^{-n+2} \quad \text{ when } d_g(x,y) \leq |z|^{-\frac{1}{2}}.
     \end{align}
     Furthermore, one has
     \begin{align}
      \label{Had:Asymp2}
          F(x,y,z) &= |z|^{\frac{n-1}{4}-\frac{1}{2}} \e^{-\sqrt{z} \, d_g(x,y)} d_g(x,y)^{-\frac{n-1}{2}} a(x,y,z)
     \end{align}
     where the function $a$ has a symbolic behaviour
     \begin{align*}
           \big|\d^\alpha_{x,y} a(x,y,z)\big| \leq C_{\alpha} d_g(x,y)^{-\alpha}, \quad \alpha \in \N
     \end{align*}
     when $d_g(x,y)\geq |z|^{-\frac{1}{2}}$. The constants are all uniform with respect to $z \in \C$, $|z| \geq c$.
\end{lemma} 
\end{section}
%
%
\begin{section}{Resolvent estimates}
This section is devoted to the proof of $L^p$ estimates on the resolvent.

\subsection{Parametrix estimates}
     
Our goal in this paragraph is to prove that the Hadamard parametrix is bounded on certain Lebesgue spaces. 
Due to the local nature of the construction, the kernel of the parametrix is supported in $U \times U$ with $U$ a small neighbourhood of a point 
$x_0 \in M$; using local coordinates, we might as well suppose that $x_0=0$ and that $U$ is a small ball in $\R^n$ of radius $\eps>0$.
\begin{thm}
\label{LpEst:HadBdness}
 The Hadamard parametrix is a bounded operator 
\begin{align}
\label{LpEst:GenHadParam}
     |z|^{s} \, T_{\rm Had}(z) : L^{p}(M) &\to L^{q}(M)
\end{align}
with a norm uniform with respect to the spectral parameter $z \in \C$, $|z| \geq 1$ when $0 \leq s \leq 1$, $p \leq 2 \leq q$
\begin{align}
\label{LpEst:Scale}
     \frac{1}{p}-\frac{1}{q}+\frac{2s}{n} = \frac{2}{n} , 
\end{align}
and 
\begin{align*}
      \min \bigg(\frac{1}{p}-\frac{1}{2},\frac{1}{2}-\frac{1}{q}\bigg) > \frac{1}{2n},
      \quad \frac{2}{n+1} < \frac{1}{p}-\frac{1}{q} \leq \frac{2}{n}.
\end{align*}
\end{thm}

\begin{center}
\begin{figure}
     \input{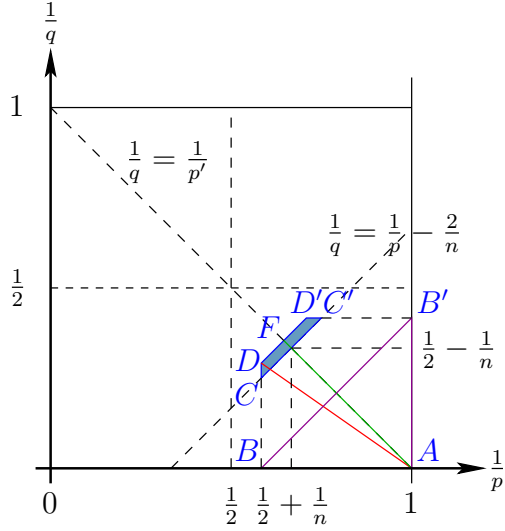}
     \caption{Admissible exponents in Theorem \ref{LpEst:HadBdness}}
     \label{fig:AdmExpTer}
\end{figure}
\end{center}

\begin{rem}
      The region of admissible exponents in Theorem \ref{LpEst:HadBdness} is the trapezium $CC'D'D$ (blue quadrilateral in figure \ref{fig:AdmExpTer})
      with vertices
      \begin{gather*}
          C=\bigg(\frac{1}{2}+\frac{1}{2n},\frac{1}{2}-\frac{3}{2n}\bigg), \quad C'=\bigg(\frac{1}{2}+\frac{3}{2n},\frac{1}{2}-\frac{1}{2n}\bigg). \\
          D=\bigg(\frac{1}{2}+\frac{1}{2n},\frac{(n-1)^2}{2n(n+1)}\bigg), \quad D'=\bigg(\frac{n^2+4n-1}{2n(n+1)},\frac{1}{2}-\frac{1}{2n}\bigg),
      \end{gather*}
       Note that
           $$ F=\bigg(\frac{1}{2}+\frac{1}{n+1},\frac{1}{2}-\frac{1}{n+1}\bigg) $$
       is the middle of $[CC']$.
\end{rem}

Let $\psi_{0} \in \D(\R)$ be supported in $[-1,1]$ and equal one on $[-\tfrac{1}{2},\tfrac{1}{2}]$, let $\psi=\psi_{0}(\cdot/2)-\psi_{0} \in \D(\R)$ 
be supported in $[-1,-\tfrac{1}{2}] \cup [\tfrac{1}{2},1]$, we consider the following dyadic partition of unity
    $$ \psi_0(r)+\sum_{\nu = 0}^{\infty} \psi(2^{-\nu}r) = 1. $$
We denote $\psi_{\nu} = \psi(2^{-\nu} \cdot)$ if $\nu \geq 1$ and decompose Hadamard's parametrix in the following way
\begin{align*}
     T_{\rm Had}(z) = T_{0}(z) + \sum_{\nu=1}^{\infty} T_{\nu}(z)
\end{align*}
where 
\begin{align*}
     T_{\nu}(z)u(x) &=\int_{M}  \chi(x,y) \psi_{\nu}\big(|z|^{\frac{1}{2}}d_g(x,y)\big)F(x,y,z) u(y) \, \dd V_g(y). 
\end{align*}
We first observe that 
\begin{align}
     d_g^2(x,y) = \frac{1}{2} G_{jk}(x,y)(x^j-y^j)(x^k-y^k) 
\end{align}
where $G(x,x)=g(x)$ for all $x \in U$. Indeed, the square of the distance function vanishes at order two at $x=y$ and its Hessian is twice the 
metric since
\begin{align*}
      \nabla^2 \big[d_g^2(\cdot,y)\big]_{x=y} (\theta,\theta) = \frac{\d^2}{\d t^2} d^2_g(\exp_y t\theta,y) \Big|_{t=0} =2 |\theta|_g^2.
\end{align*} 
In particular, in the chart $U$ the geodesic distance is equivalent to the Euclidean distance, i.e. 
\begin{align}
\label{LpEst:EquivDist}
      \frac{1}{C}|x-y|\leq d_g(x,y) \leq C |x-y|.
\end{align}
when $x,y \in U$.

Let us begin by examining the term $T_0(z)$: using the bound \eqref{Had:Asymp1} in Lemma \ref{Had:FDescr}, we majorize the kernel 
by a constant times 
     $$ |x-y|^{-n+2}1_{|x-y| \leq C^{-1}|z|^{-1/2}}. $$ 
Using Young's inequality when $0 \leq 1/p-1/q<2/n$ and Hardy-Littlewood-Sobolev's inequality when $1/p-1/q=2/n$,
and scaling the estimates, bearing in mind the condition \eqref{LpEst:Scale}, we get
\begin{align}
\label{LpEst:T0}
      |z|^{s} \|T_{0}(z)\|_{\mathcal{L}(L^p,L^q)} \lesssim |z|^{\frac{n-2}{2}+\frac{n}{2}\big(\frac{2s}{n}+\frac{1}{p}-\frac{1}{q}-1\big)}
      =1
\end{align}
when $0 \leq 1/p-1/q \leq 2/n$ excluding the endpoints $(0,2/n)$ and $(2/n,0)$ of the segment $1/p-1/q=2/n$.
Now we deal with the other terms: on the support of the kernel of $T_{\nu}(z)$ we have $d_g(x,y) \geq 2^{\nu-1}|z|^{-1/2}$
 and if we use the asymptotic description \eqref{Had:Asymp2} in Lemma \ref{Had:FDescr},  the operator $T_{\nu}(z)$ assumes the form
\begin{multline*}
     T_{\nu}(z)u(x) = \\ |z|^{\frac{n-1}{4}-\frac{1}{2}} \int_M  \e^{-\sqrt{z} d_g(x,y)} \frac{a(x,y,z)}{d_g(x,y)^{\frac{n-1}{2}}}  
      \chi(x,y)\psi_{\nu}\big(|z|^{\frac{1}{2}}d_g(x,y)\big) u(y) \, \dd V_g(y). 
\end{multline*}
If we denote
    $$ A(x,y,z) = \frac{\e^{-\re \sqrt{z} \, d_g(x,y)} }{d_g(x,y)^{\frac{n-1}{2}}}  a(x,y,z) \chi(x,y) \sqrt{\det g(y)} $$ 
then we have 
\begin{align}
\label{LpEst:Symbolic}
     \d^{\alpha}_{x,y} A = \mathcal{O}\Big(d_g(x,y)^{-\frac{n-1}{2}-|\alpha|}\Big)
\end{align}
uniformly in $z$. 

We need to distinguish two further cases: $\arg z \in [-\pi/2,\pi/2]$  where the exponential is uniformly decaying and 
$\arg z \notin [-\pi/2,\pi/2]$ where the exponential has an oscillating behaviour. Let us begin by the first case which is straightforward:
thanks to \eqref{LpEst:EquivDist} the kernel of $T_{\nu}(z)$  may be bounded by a constant times
    $$ |z|^{\frac{n-2}{2}} 2^{-\frac{n-1}{2}\nu} \e^{-\frac{c}{\sqrt{2}}|z|^{1/2}|x-y|}  $$
and Young's inequality provides the bound
\begin{align*}
      |z|^{s} \|T_{\nu}(z)\|_{\mathcal{L}(L^p,L^q)} \lesssim 2^{-\frac{n-1}{2}\nu}\e^{-c2^{\nu}}.
\end{align*}
This (together with \eqref{LpEst:T0}) can be summed as a geometric series 
\begin{align*}
        |z|^{s} \, \|T_{\rm Had}(z)\|_{\mathcal{L}(L^p,L^q)} &\leq \sum_{\nu=0}^{\infty} |z|^s \, \|T_{\nu}(z)\|_{\mathcal{L}(L^p,L^q)} \\
        &\lesssim \sum_{\nu=0}^{\infty} 2^{-\frac{n-1}{2}\nu} \lesssim 1
\end{align*}
so it only remains to consider the case where $\arg z \notin [-\pi/2,\pi/2]$.

\subsubsection*{The case $\frac{1}{p}-\frac{1}{q}>\frac{1}{2}+\frac{1}{2n}$}

We majorize the kernel by a constant times
     $$ |z|^{\frac{n-2}{2}} 2^{-\nu \frac{n-1}{2}} 1_{|x-y| \leq c^{-1}2^{\nu}|z|^{-1/2}} $$ 
and using Young's inequality we get
\begin{align*}
      |z|^{s} \|T_{\nu}(z)\|_{\mathcal{L}(L^p,L^q)} \lesssim 2^{\nu\big(\frac{n+1}{2}+n\big(\frac{1}{q}-\frac{1}{p}\big)\big)}.
\end{align*}
As above we can sum a geometric series and obtain the boundedness of the $|z|^s(T_{\rm Had}(z)-T_0(z))$ as long as the exponent is positive, i.e. when
   $$ \frac{1}{p}-\frac{1}{q}>\frac{1}{2}+\frac{1}{2n}. $$
This region is in the plane $(1/p,1/q)$ the triangle $ABB'$  (delimited by the purple lines in figure \ref{fig:AdmExpTer}).

\subsubsection*{The other cases}

We first observe that, because of the support properties of the kernel, the only terms $T_{\nu}(z)$ which are nonzero are those for which
$2^{\nu}|z|^{-1/2}$ is bounded, and in fact even small if one is prepared to shrink the size of $U$. We consider the following rescaled 
distance function
\begin{align*}
     \varrho_{\nu}(x,y,z) &= |z|^{\frac{1}{2}}2^{-\nu}d_g\big(2^{\nu}|z|^{-\frac{1}{2}}x,2^{\nu}|z|^{-\frac{1}{2}}y\big)  \\
     &=\sqrt{\frac{1}{2}G_{jk}\Big(2^{\nu}|z|^{-\frac{1}{2}}x,2^{\nu}|z|^{-\frac{1}{2}}y\Big)(x^j-y^j)(x^k-y^k)} \\
     &\simeq |x-y|
\end{align*}
and the following rescaled operator
\begin{align*}
     \widetilde{T}_{\nu}(z)u(x) &=|z|^{\frac{n-2}{2}} 2^{-\nu\frac{n-1}{2}} \int  \e^{-i 2^{\nu} \phi_{\nu}(x,y,z)}  a_{\nu}(x,y,z) u(y)  \, \dd y
\end{align*}
which was obtained by the contraction $(x,y) \to \big(2^{\nu}|z|^{-\frac{1}{2}}x,2^{\nu}|z|^{-\frac{1}{2}}y\big)$.
The phase is given by 
\begin{align*}
     \phi_{\nu}(x,y,z) &= 2^{-\nu} \im \sqrt{z} \,  d_g\big(2^{\nu}|z|^{-\frac{1}{2}}x,2^{\nu}|z|^{-\frac{1}{2}}y\big) \\ 
     &= \sin \bigg(\frac{\arg z}{2}\bigg) \, \varrho_{\nu}(x,y,z) \nonumber
\end{align*}
and the amplitude by
\begin{align*}
     a_{\nu}(x,y,z) &= |z|^{-\frac{n-1}{4}} 2^{\nu\frac{n-1}{2}} A\big(2^{\nu}|z|^{-\frac{1}{2}}x,2^{\nu}|z|^{-\frac{1}{2}}y,z\big) \psi\big(\varrho_{\nu}(x,y,z)\big).
\end{align*}
With those notations, we have
\begin{align*}
      T_{\nu}(z)u(x) = \big(2^{-\nu}|z|^{\frac{1}{2}}\big)^{-n} \widetilde{T}_{\nu}(z)v\big(2^{-\nu}|z|^{\frac{1}{2}}x\big), 
      \quad v(x) = u\big(2^{\nu}|z|^{-\frac{1}{2}}x\big).
\end{align*}
Because of \eqref{LpEst:Symbolic} and because $2^{\nu}|z|^{-1/2} \lesssim 1$, both phase and amplitude are uniformly bounded as well as all their derivatives.  Besides, the amplitude $a_\nu(\cdot,\cdot,z)$ is compactly supported in $\R^n \times \R^n$ (with possibly a very large support) but the set
    $$ \big\{x-y : (x,y) \in \supp a_{\nu}(\cdot,\cdot,z) \big\} $$
is contained in a fixed compact set of $\R^n$. 

This is an oscillatory integral operator, whose phase is (a rescaled version of) the Riemannian geodesic distance, and the previous observation 
shows that part of the assumptions in Remark \ref{Osc:RemSupport} are fulfilled. Besides we have
\begin{align*}
     \frac{\d^2 \phi_{\nu}}{\d x \d y}  = \underbrace{\sin \bigg(\frac{\arg z}{2}\bigg)}_{|\cdots| \geq \sqrt{2}/2} 
     \frac{\d^2}{\d x \d y}\sqrt{g_{jk}(x_0)(x^j-y^j)(x^k-y^k)} + \mathcal{O}(\eps)
\end{align*}
this means that the non-degeneracy condition \eqref{Osc:CorankEquiv} is satisfied if $\eps$ is taken small enough. Corollary \ref{Osc:NonDegClry} applies
\begin{align}
      \|\widetilde{T}_{\nu}(z)u\|_{L^{p'}(M)} \lesssim |z|^{\frac{n-2}{2}} 2^{-\nu \frac{n-1}{2}} 2^{-\nu (n-1)/p'} \|u\|_{L^p(M)}
\end{align}
and after rescaling, we obtain
\begin{align}
\label{LpEst:NonDegEst}
      \nonumber
      |z|^{s}\|T_{\nu}(z)\|_{\mathcal{L}(L^p,L^{p'})} &\lesssim (|z|^{1/2}2^{-\nu})^{n\big(\frac{1}{p}-\frac{1}{p'}-1\big)} |z|^{\frac{n-2}{2}}
      2^{-\nu \frac{n-1}{2}} 2^{-\nu (n-1)/p'} \\ &\lesssim 2^{\nu\big(\frac{n+1}{p'} - \frac{n-1}{2}\big)}.
\end{align}
Again, we can sum those estimates into a geometric series provided that
     $$ \frac{1}{p} > \frac{1}{2}+\frac{1}{n+1}.$$
This shows that the operator $|z|^s(T_{\rm Had}(z)-T_0(z))$ is uniformly bounded on the semi-open segment $[AF)$ in figure \ref{fig:AdmExpTer}. 

Since the distance function $r=d_{g}(\cdot,y)$ satisfies the eikonal equation   
    $$ |\dd r|_{g} =1 $$
its differential parametrizes a piece of the cosphere bundle, thus
\begin{align}
     \Sigma_{y_{0}} = \big\{\dd_{x} d_{g}(x,y_{0}),  x \in U\big\} \subset S^*_{x_{0}}M
\end{align}
has everywhere non-vanishing Gaussian curvature. This is also true when $x$ and $y$ are interchanged because
of the symmetry of the function. This remains true after rescaling and it is not hard to see that there is 
a uniform lower bound on the Gaussian curvature when $\eps$ is small enough.   
This means that the Carleson-Sj\"olin condition is satisfied (in its uniform variant mentioned in Remark \ref{Osc:RemSupport})
and Corollary \ref{Osc:CSThm} applies
\begin{align}
      \|\widetilde{T}_{\nu}(z)u\|_{L^q(M)} \lesssim |z|^{\frac{n-2}{2}} 2^{-\nu \frac{n-1}{2}} 2^{-\nu n/q} \|u\|_{L^p(M)}
\end{align}
if $q = (n+1)p'/(n-1)$ and $1 \leq p \leq 2$. After rescaling we finally obtain
\begin{align}
\label{LpEst:EstAC}
      \nonumber
      |z|^{s} \|T_{\nu}(z)\|_{\mathcal{L}(L^p,L^q)} &\lesssim (|z|^{1/2}2^{-\nu})^{n\big(\frac{1}{p}-\frac{1}{q}-1\big)} |z|^{\frac{n-2}{2}}
      2^{-\nu \frac{n-1}{2}}  2^{-\nu n/q} \\ &\lesssim 2^{-\nu \big(\frac{n}{p}-\frac{n+1}{2}\big)}.
\end{align}
Those estimates can be summed into a geometric series when
  $$ \frac{1}{p} > \frac{1}{2}+\frac{1}{2n}. $$
This shows that the Hadamard parametrix is bounded on the red segment $[AC)$ in figure \ref{fig:AdmExpTer}.

If we sum up the results obtained so far: we have shown the boundedness of $|z|^{s}(T_{\rm Had}(z)-T_0(z))$ on the segments $[AB)$ (by Young's inequality),
 $[AD)$ (by the Carleson-Sj\"olin theory), $[AF)$ (by the non-degeneracy corollary \ref{Osc:NonDegClry}) in figure \ref{fig:AdmExpTer}. 
By interpolation and duality, this shows that $|z|^{s}(T_{\rm Had}(z)-T_0(z))$ is bounded on the pentagon $ABDD'B'$ in figure \ref{fig:AdmExpTer}.
This ends the case $\arg z \notin [-\pi/2,\pi/2]$ and hence the boundedness  of Hadamard's parametrix \eqref{LpEst:GenHadParam}.
This completes the proof of Theorem \ref{LpEst:HadBdness}.

\begin{rem}
     In fact, the computation made to obtain the bound on the segment $[A,F)$ is redundant since the bound can be obtained by interpolation and duality 
     from the bound on the segment $[A,D)$.
\end{rem}

\subsection{Additional bounds}

If one is prepared to loose some powers of the spectral paramater $|z|$, one can improve the range of allowed exponents.
Indeed, the estimate \eqref{LpEst:EstAC} can still be summed on $2^{\nu}\leq|z|^{1/2}$ when $\frac{1}{p} < \frac{1}{2}+ \frac{1}{2n}$
\begin{align*}
        |z|^{s} \, \|T_{\rm Had}(z)\|_{\mathcal{L}(L^p,L^q)} &\leq \sum_{\nu=0}^{\infty} |z|^s \, \|T_{\nu}(z)\|_{\mathcal{L}(L^p,L^q)} \\
        &\lesssim \sum_{\nu \geq 0 \, : \, 2^{\nu}|z|^{-1/2} \leq 1} 2^{\nu\big(\frac{n+1}{2} - \frac{n}{p}\big)} 
        \lesssim |z|^{\big(\frac{n+1}{4} - \frac{n}{2p}\big)}.
\end{align*}
Note that when $q=\frac{n+1}{n-1}p'$ we have
     $$ (n+1)\bigg(\frac{1}{p}-\frac{1}{q}\bigg) = \frac{2n}{p} - (n-1) $$ 
and therefore
     $$  |z|^{-s+\big(\frac{n+1}{4} - \frac{n}{2p}\big)} = |z|^{\frac{n-1}{4}\big(\frac{1}{p}-\frac{1}{q}\big)-\frac{1}{2}}. $$
The values of $(1/p,1/q)$ obtained in this estimate correspond to the semi-open segment $(DE]$ on the figure \ref{fig:AdmExpQuar}.
Similarly, the estimate \eqref{LpEst:NonDegEst} can still be summed when $\frac{1}{p}<\frac{1}{2}+\frac{1}{n+1}$
\begin{align*}
        |z|^{s} \, \|T_{\rm Had}(z)\|_{\mathcal{L}(L^p,L^{p'})} &\leq \sum_{\nu=0}^{\infty} |z|^s \, \|T_{\nu}(z)\|_{\mathcal{L}(L^p,L^{p'})} \\
        &\lesssim \sum_{\nu \geq 0 \, : \, 2^{\nu}|z|^{-1/2} \leq 1} 2^{\nu\big(\frac{n+1}{p'} - \frac{n-1}{2}\big)} 
        \lesssim |z|^{\big(\frac{n+1}{2p'} - \frac{n-1}{4}\big)}.
\end{align*}
Note that we have
     $$ (n+1)\bigg(\frac{1}{p}-\frac{1}{p'}\bigg) = -\frac{2(n+1)}{p'} + (n+1) $$ 
and therefore we retrieve the same bound as above
     $$  |z|^{-s+\big(\frac{n+1}{2p'} - \frac{n-1}{4}\big)} = |z|^{\frac{n-1}{4}\big(\frac{1}{p}-\frac{1}{q}\big)-\frac{1}{2}}. $$
The values of $(1/p,1/q)$ obtained in this estimate correspond to the semi-open segment $(FG]$ on the figure \ref{fig:AdmExpQuar}.
\begin{center}
\begin{figure}
     \input{ExponentsBis_t.tex}
     \caption{Admissible exponents in Theorem \ref{LpEst:HadBdnessBis}}
     \label{fig:AdmExpQuar}
\end{figure}
\end{center}

By duality and interpolation, we get the following theorem.
\begin{thm}
\label{LpEst:HadBdnessBis}
 The Hadamard parametrix is a bounded operator 
\begin{align}
      T_{\rm Had}(z) : L^{p}(M) &\to L^{q}(M)
\end{align}
when $1 \leq p \leq 2 \leq q$ and 
   $$ \frac{1}{p}-\frac{1}{q} < \frac{2}{n+1}, \quad  \frac{n-1}{n+1} \, \frac{1}{p'} \leq  \frac{1}{q} \leq \frac{n+1}{n-1}  \,\frac{1}{p'}, $$
with a norm bounded by
\begin{align}
\label{LpEst:LpBoundsBis}
      \|T_{\rm Had}(z)\|_{\mathcal{L}(L^p,L^q)} \leq C |z|^{\frac{n-1}{4}\big(\frac{1}{p}-\frac{1}{q}\big)-\frac{1}{2}}.
\end{align}
\end{thm}
The extended set of admissible exponents is the pentagon $DEGE'D'$, represented by the union of the green and red sets in figure \ref{fig:AdmExpQuar}.
These estimates (or rather the endpoint $p=2$, $q=2(n+1)/(n-1)$) were proved by Sogge in \cite{Sog}
\begin{align*}
      \|T_{\rm Had}(z)u\|_{L^{\frac{2(n+1)}{n-1}}(M)} \lesssim |z|^{-\frac{n+3}{4(n+1)}}  \|u\|_{L^2(M)}. 
\end{align*}
\begin{center}
\begin{figure}[h]
     \input{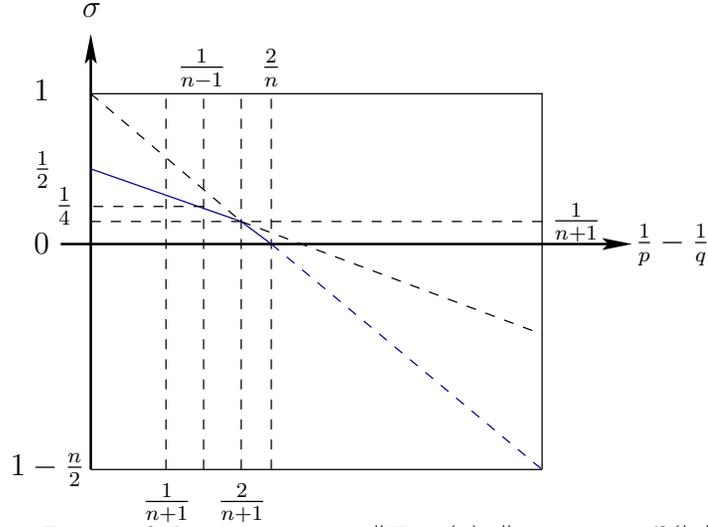}
     \caption{Decay of the parametrix: $ \|T_{\rm Had}(z)u\|_{\mathcal{L}(L^p,L^q)} = \mathcal{O}(|z|^{-\sigma})$}
     \label{fig:DecayExp}
\end{figure}
\end{center}      
The combination of Theorems \ref{LpEst:HadBdness} and \ref{LpEst:HadBdnessBis} gives the following bound on the parametrix
\begin{align}
\label{LpEst:AddDecayParam}
      \|T_{\rm  Had}(z)u\|_{\mathcal{L}(L^p,L^q)}  \leq C |z|^{-\sigma}
\end{align}
in the polygon $CDEGE'D'C'$ (except for the segments $[CD]$, $[DD']$ and $[D'C']$), where the order $\sigma$ is a piecewise 
linear function of $d= 1/p-1/q$
\begin{align}
\label{LpEst:Gain}
     \sigma = \begin{cases}   \displaystyle -\frac{n-1}{4} d+ \frac{1}{2}  &\displaystyle \text{ when } 0 \leq d \leq \frac{2}{n+1}  \\ 
     \displaystyle  -\frac{n}{2} d + 1 & \displaystyle \text{ when } \frac{2}{n+1} \leq d \leq 1 \end{cases}.
\end{align}
The graph of this function is represented in figure \ref{fig:DecayExp}.

We finish by observing that as a function of the $y$ variable, the kernel of the $T_{\rm Had}(z)-T_0(z)$ is $L^r$-integrable 
when $r<2n/(n-1)$ since 
\begin{align*}
     \big|\big(T_{\rm Had}(z)-T_0(z)\big)(x,y)\big| \lesssim |z|^{\frac{n-3}{4}} d_g(x,y)^{-\frac{n-1}{2}}.
\end{align*}
By Young's inequality, this provides the trivial bound
\begin{align*}
     \|(T_{\rm Had}(z)-T_0(z))u\|_{L^{\infty}(M)} \leq |z|^{\frac{n-3}{4}} \|u\|_{L^p(M)}, \quad \frac{1}{p}<\frac{1}{2}+\frac{1}{2n}, 
\end{align*}
and by interpolation, we obtain the following additional bounds 
\begin{multline*}
      \|(T_{\rm Had}-T_0)(z)u\|_{L^q(M)} \leq C
          |z|^{\frac{n-3}{4}-\frac{n}{2q}}  \|u\|_{L^p(M)} \\ \text{ if }  \frac{1}{q}  \leq \frac{n-1}{n+1} \, \frac{1}{p'} \quad \text{ and } \quad
          \frac{1}{p} < \frac{1}{2}+\frac{1}{2n}.
\end{multline*}
Since $T_0(z)$ satisfies similar estimates with better decay in $|z|$ when $1/p \leq 1/2+1/2n$, the previous estimates remain true
for the Hadamard parametrix. These estimates with their dual provide  the following additional boundedness properties of the Hadamard parametrix.
\begin{lemma}
\label{LpEst:AddEstLemHad}
      Let $1 \leq p \leq 2 \leq q$ be such that $1/p-1/q \leq 2/n$, the Hadamard parametrix satisfies the following bounds 
      \begin{align*}
      \|T_{\rm Had}(z)\|_{\mathcal{L}(L^q,L^p)} \leq C
      \begin{cases}
          |z|^{\frac{n-3}{4}-\frac{n}{2q}}  & \text{ if }  \frac{1}{q}  \leq \frac{n-1}{n+1} \, \frac{1}{p'} \text{ and } 
          \frac{1}{p} < \frac{1}{2}+\frac{1}{2n} \\    
          |z|^{\frac{n-3}{4}-\frac{n}{2p'}} & \text{ if }   \frac{1}{q} \geq \frac{n+1}{n-1}  \,\frac{1}{p'} \text{ and } 
          \frac{1}{q} > \frac{1}{2}-\frac{1}{2n}
       \end{cases}.
\end{align*}
\end{lemma}
The range of exponents covered by this theorem is the union of the two violet trapezia in figure \ref{fig:AdmExpQuar}.
\begin{center}
\begin{figure}[h]
     \input{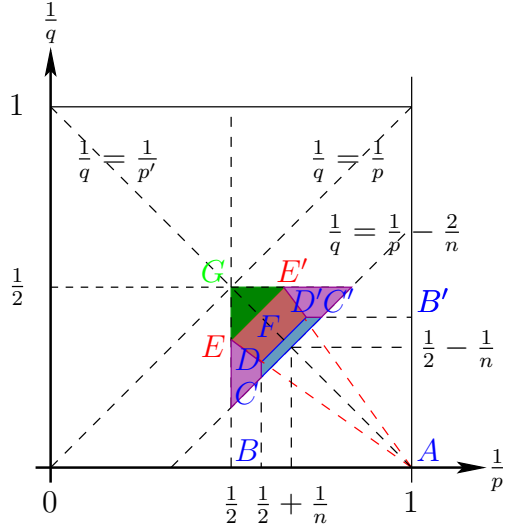}
     \caption{Extended set of admissible exponents}
     \label{fig:AdmExpQuar}
\end{figure}
\end{center}
In particular, we have%
\begin{align}
\label{LpEst:CrucialEst}
      \begin{split}
           \|T_{\rm Had}(z)u\|_{L^{\frac{2n}{n-2}}(M)} &\lesssim |z|^{-\frac{1}{4}}  \|u\|_{L^2(M)} \\
           \|T_{\rm Had}(z)u\|_{L^2(M)} &\lesssim |z|^{-\frac{1}{4}}  \|u\|_{L^{\frac{2n}{n+2}}(M)}. 
      \end{split}
\end{align}
These estimates will prove useful in the final argument proving the resolvent estimates when dealing with the remainders.

\subsection{Remainder estimates}

We now deal with the remainder term $S(z)$: the results are similar than for the Hadamard parametrix (in fact the range of admissible exponents
is larger) but the proof is simpler due to the lack of singularities since the kernel is supported away from the diagonal in $M \times M$. 
\begin{lemma}
\label{LpEst:RemainderEstLem}
The operator $S(z)$ defined in \eqref{Had:ParmaIdentity} is a bounded operator 
\begin{align}
      S(z) : L^{p}(M) &\to L^{q}(M)
\end{align}
when $1 \leq p \leq 2 \leq q$ with a norm bounded by
\begin{align*}
      \|S(z)\|_{\mathcal{L}(L^p,L^q)} \leq C
      \begin{cases}
          |z|^{\frac{n-1}{4}-\frac{n}{2q}} & \text{ if }  \frac{1}{q}  \leq \frac{n-1}{n+1} \, \frac{1}{p'} \\ 
          |z|^{\frac{n-1}{4}\big(\frac{1}{p}-\frac{1}{q}\big)} & \text{ if }  \frac{n-1}{n+1} \, \frac{1}{p'} \leq  \frac{1}{q} 
          \leq \frac{n+1}{n-1}  \,\frac{1}{p'}, \\    
          |z|^{\frac{n-1}{4}-\frac{n}{2p'}} & \text{ if } \frac{1}{q} \geq \frac{n+1}{n-1}  \,\frac{1}{p'}
       \end{cases}.
\end{align*}
\end{lemma}
\begin{proof}
      The kernel of $S_2(z)$ is uniformly bounded by $|z|^{-1/2}$ because of \eqref{Had:RemGood}, hence $S_2(z)$ is a bounded operator
      from $L^p(M)$ to $L^q(M)$ with norm bounded by $|z|^{-1/2}$.  For this part of the operator we actually get better bounds than expected since 
      $|z| \geq \delta$. We now concentrate on the part $S_1(z)$ of the operator: because of Lemma \ref{Had:FDescr} 
      and the support of $\grad_g \chi(\cdot,y)$, the kernel of this operator looks like
      \begin{align*}
            |z|^{\frac{n-1}{4}} \e^{-\sqrt{z} \, d_g(x,y)} b(x,y,z) 
      \end{align*}
      where $b$ is a smooth function. 
      
      When $\arg z \in [-\pi/2,\pi/2]$, Young's inequality provides a better bound
           $$ \|S_1(z)u\|_{L^q(M)} \lesssim |z|^{-\frac{n+1}{4}+\frac{n}{2}\big(\frac{1}{p}-\frac{1}{q}\big)}\e^{-c|z|^{1/2}} \|u\|_{L^p(M)}. $$
      When $\arg z \notin [-\pi/2,\pi/2]$, the bound follows from Young's inequality, Theorems \ref{Osc:DegL2Thm} and 
      \ref{Osc:CSThm}, and duality and interpolation.
\end{proof}

\subsection{Resolvent estimates}

Before actually turning our attention to resolvent estimates, let us first patch our local constructions into a global parametrix. 
We cover the manifold $M$ with finitely many open geodesic balls $U_j$ of radius $\eps$ contained in charts of $M$.
We consider a partition of unity 
    $$ 1 = \sum_{j=1}^J \chi_j^2, \quad \supp \chi_j \subset U_j $$
subordinated to the covering $(U_j)_{1 \leq j \leq J}$, and define $T_{\rm Had}^{(j)}(z)$ to be the local Hadamard parametrix
constructed in section \ref{Sec:Had} using  $\chi_j(x) \otimes \chi_j(y)$ as the cutoff function. By a slight abuse of notation,
we keep $T_{\rm Had}(z)$ to denote the global parametrix
   $$ T_{\rm Had}(z) = \sum_{j=1}^J T_{\rm Had}^{(j)}(z) $$
and similarly for the remainder
   $$ S(z) = \sum_{j=1}^J  S^{(j)}(z). $$
Having constructed those operators, we indeed get a global right parametrix
\begin{align}
\label{LpEst:ParamIdentity}
   (-\Delta_g+z)T_{\rm Had}(z) = \id +S(z). 
\end{align}
As a finite sum of operators, $T_{\rm Had}(z)$ shares the same continuity properties as the local parametrices, and this is also true of $S(z)$.

We are now ready to prove the resolvent estimates. We use ideas coming from \cite{KT} and also \cite{DSF,DSF1} to deal with the remainder terms.
\subsubsection{$L^2$ estimates}
Let $0=\lambda_0 \leq \lambda_1 \leq \dots \leq \lambda_k \leq \dots$ be the sequence of eigenvalues of the (positive)
Laplacian $-\Delta_g$ repeated with multiplicity and $(\psi_k)_{k=0}^{\infty}$ be the corresponding sequence of unit eigenfunctions. We denote
     $$ \widehat{f}(k) = \int_M f \,  \overline{\psi_k} \, \dd V_g $$
the Fourier coefficient of the function $f$ corresponding to the eigenfunction~$\psi_k$. 
\begin{lemma}
\label{LpEst:L2ResEst}
     Let $(M,g)$ be a compact Riemannian manifold (without boundary) of dimension $n \geq 3$,
     for all $z\in \C \setminus \R_-$  we have 
    \begin{align*}
         \|u\|_{L^2(M)} \leq |z|^{-\frac{1}{2}} (\re \sqrt{z})^{-1} \|(-\Delta_g+z)u\|_{L^2(M)}.
    \end{align*}
\end{lemma}
\begin{proof}
     We start with the factorization
          $$ \lambda_k + z = \big(\sqrt{\lambda_k}+i\sqrt{z}\big)\big(\sqrt{\lambda_k}-i\sqrt{z}\big) $$
     which implies 
          $$ |\lambda_k+z| \ge \big(\lambda_k+|z|\big)^{\frac{1}{2}} \re \sqrt{z}.  $$ 
     If $f=(-\Delta_g+z)u$, using Fourier series, we get
      \begin{align*}
           \|u\|^2_{L^2(M)} = \sum_{k=0}^{\infty} \frac{|\widehat{f}(k)|^2}{|z+\lambda_k|^2}  
           \lesssim |z|^{-1} \big|\re \sqrt{z}\big|^{-2} \|f\|_{L^2(M)}^2. 
      \end{align*}
      This completes the proof of the lemma.
\end{proof}
\begin{rem}
     When $\re \sqrt{z} \geq \delta$, we deduce the following improved $L^2$ bound
      \begin{align}
      \label{LpEst:ImprovedL2}
         \|u\|_{L^2(M)} \lesssim |z|^{-\frac{1}{2}}  \|(-\Delta_g+z)u\|_{L^2(M)}.
    \end{align}
    which will prove to be of crucial importance to estimate the remainder terms in the proof of Theorem \ref{Intro:MainThm}.
    Unfortunately, such an improved bound only holds when the spectral parameter is outside of a parabola. Indeed suppose
    that the estimate \eqref{LpEst:ImprovedL2} holds then testing it with unit eigenfunctions $\psi_k$ implies
         $$ |\lambda_k+z| \geq c|z|^{\frac{1}{2}} $$
    and the parameter $z$ has to remain outside of the envelope of the curves $|\lambda_k+z| = c|z|^{\frac{1}{2}}$, $k \in \N$.
    In particular, if we choose $z=-\lambda_k+iy$ then we have
         $$ y^4 \geq c^4(\lambda_k^2+y^2),  $$
    and $|y| \geq \sqrt{c \lambda_k}$.
         
    Besides, if one is interested in proving resolvent estimates when $z$ lies in $\widetilde{\Xi}_{\delta}$,
    then difficulties arise when the parameter $z$ lies on the lines $\im z = \pm \delta$.
    This means that the square root of $z$ lies on the part of hyperbola    
         $$ \sqrt{z} = \lambda^{-1} + i \frac{\delta}{2} \lambda, \quad \lambda \geq \sqrt{\frac{2}{\delta}} $$
    (see figure \ref{fig:SqrtAdm}) and this implies
         $$  |z|^{-\frac{1}{2}} \big|\re \sqrt{z}\big|^{-1} \simeq 1. $$
     In other words, the $L^2$ estimate is weak 
     \begin{align*}
         \|u\|_{L^2(M)} \lesssim  \|(-\Delta_g+z)u\|_{L^2(M)}
    \end{align*}
    in terms of the decay with respect to the size of the spectral parameter $|z|\simeq \lambda^2$. 
\end{rem}
\begin{center}
\begin{figure}[h]
    \input{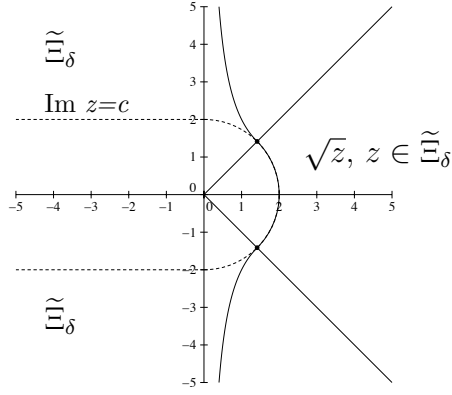}
     \caption{Values of the spectral parameters $z \in \widetilde{\Xi}_{\delta}$ and $\sqrt{z}$}
     \label{fig:SqrtAdm}
\end{figure}
\end{center}
\subsubsection{$L^{\frac{2n}{n+2}}-L^{\frac{2n}{n-2}}$ estimate}
We now prove the $L^{\frac{2n}{n+2}}-L^{\frac{2n}{n-2}}$ estimate of Theorem \ref{Intro:MainThm}.
We start from 
\begin{align*}
      u = T_{\rm Had}(z)(-\Delta_g+z)u + v
\end{align*}
where 
     $$ v= \big(\id-T_{\rm Had}(z)(-\Delta_g+z)\big)u. $$
Note that by \eqref{LpEst:ParamIdentity} we have
\begin{align}
\label{LpEst:RemainderIdentity}
     (-\Delta_g+z)v =-S(z)(-\Delta_g+z)u.
\end{align}
We start our estimations
\begin{align*}
     \|u\|_{L^{\frac{2n}{n-2}}(M)} \leq \| T_{\rm Had}(z)(-\Delta_g+z)u \|_{L^{\frac{2n}{n-2}}(M)} + \|v\|_{L^{\frac{2n}{n-2}}(M)} 
\end{align*}
and by the boundedness properties of the Hadamard parametrix given in Theorem \ref{LpEst:HadBdness}, we have
\begin{align*}
     \|u\|_{L^{\frac{2n}{n-2}}(M)} \leq \| (-\Delta_g+z)u \|_{L^{\frac{2n}{n+2}}(M)} + \|v\|_{L^{\frac{2n}{n-2}}(M)}. 
\end{align*}
To estimate the second right hand-side term, we need the following lemma.
\begin{lemma}
\label{LpEst:LpL2EstLem}
      Let $(M,g)$ be a compact Riemannian manifold (without boundary) of dimension $n \geq 3$, and let $\delta$ be a positive number.
      For all $z\in \Xi_{\delta}$  we have 
     \begin{align}
     \label{LpEst:RemainderEst}
           \|v\|_{L^{\frac{2n}{n-2}}(M)} \leq C |z|^{-\frac{1}{4}} \|(-\Delta_g+z)v\|_{L^2(M)}.
     \end{align}
\end{lemma}
\begin{proof}
      We use the adjoint of Hadamard's parametrix as a left parametrix
           $$ v = T_{\rm Had}(\bar{z})^*(-\Delta_g+z)v-S(\bar{z})^* v$$
      and get  the following estimate by the boundedness properties of the Hadamard parametrix given in Theorem \ref{LpEst:HadBdnessBis}
      --- more precisely \eqref{LpEst:CrucialEst} ---, and of the remainder given in Lemma \ref{LpEst:RemainderEstLem} 
      \begin{align*}
            \|v\|_{L^{\frac{2n}{n-2}}(M)} \lesssim |z|^{-\frac{1}{4}} \|(-\Delta_g+z)v\|_{L^2(M)} 
            + |z|^{\frac{1}{4}}  \|v\|_{L^2(M)}.
      \end{align*}    
      The proof of \eqref{LpEst:RemainderEst} is completed once one has controlled the second right-hand side term using the bound 
      obtained in Lemma \ref{LpEst:L2ResEst}
           $$ \|v\|_{L^2(M)} \lesssim  |z|^{-\frac{1}{2}} \|(-\Delta_g+z)v\|_{L^2(M)}. $$
      This ends the proof of the lemma.
\end{proof}
Using the previous lemma and the identity \eqref{LpEst:RemainderIdentity}, we get
\begin{align*}
     \|u\|_{L^{\frac{2n}{n-2}}(M)} &\lesssim \|(-\Delta_g+z)u \|_{L^{\frac{2n}{n+2}}(M)} 
     + |z|^{-\frac{1}{4}} \|S(z)(-\Delta_g+z)u\|_{L^2(M)}. 
\end{align*}
It is now a matter of using the boundedness properties of the remainder $S(z)$ given in Lemma \ref{LpEst:RemainderEstLem} to conclude that
\begin{align*}
     \|u\|_{L^{\frac{2n}{n-2}}(M)} &\lesssim \|(-\Delta_g+z)u \|_{L^{\frac{2n}{n+2}}(M)}.
\end{align*}
This completes the proof of Theorem  \ref{Intro:MainThm}.
\end{section}
%
%
\begin{section}{Relation with $L^p$ Carleman estimates}
The aim of this section is to illustrate the relations between resolvent estimates and analogues of the $L^p$ Carleman estimates obtained in the Euclidean case
by Jerison and Kenig \cite{JK} (for logarithmic weights) or by Kenig, Ruiz and Sogge \cite{KRS} (for linear weights). The proof of $L^2$ Carleman 
estimates for limiting Carleman weights in \cite{DKSaU}  is based on integration by parts and can not be used in the $L^p$ setting. 
However in \cite{DKSa}, we were able to prove such estimates following an idea of Jerison \cite{Jer}, (see also \cite[Section 5.1]{Sog}) using spectral
cluster estimates of Sogge \cite{Sog}. Here we present an alternative argument following Kenig, Ruiz and Sogge \cite{KRS} and Shen \cite{Shen} and reducing the Carleman estimate
to resolvent estimates. 

Limiting Carleman weights on an open Riemannian manifold $(N,g)$ are smooth functions $\phi$ such that $\dd \phi \neq 0$ on $N$ and 
\begin{align}
\label{Carl:BracketCond}
     \big\{\overline{p_{\phi}},p_{\phi}\big\}(x,\xi) = 0, \quad (x,\xi) \in p_{\phi}^{-1}(0) \subset T^*N
\end{align}
if $p_{\phi}$ denotes the following symbol
     $$ p_{\phi}(x,\xi)=|\xi|_g^2-|\dd \phi|^2_g+2i\langle \dd \phi,\xi\rangle_g. $$
This a conformally invariant property. This notion was introduced in \cite{KSU} which dealt with Calder\'on's inverse conductivity problem 
with partial data in dimension $n\geq 3$; this article was concerned with the isotropic setting and took advantage of replacing classical
linear weights by logarithmic weights (which are examples of limiting Carleman weights in the Euclidean setting). 
The motivation for introducing such weights in the context of general Riemannian manifolds lies in the application to inverse problems \cite{KSU,DKSaU,DKSa}
where solutions to the Schr\"odinger equation  with opposite exponential behaviours are constructed on manifolds using complex geometrical optics. 

It was proved in \cite{DKSaU,LS} that the condition \eqref{Carl:BracketCond} is in fact stringent. Simply connected Riemannian manifolds which admit
limiting Carleman weights are those which are conformal to one having a parallel field (i.e. simultaneously a Killing and a gradient field). In particular, such manifolds
are locally conformal to the product of an Euclidean interval and another open Riemannian manifold. The results in \cite{DKSaU,DKSa} are concerned with 
the simpler situation of manifolds conformal to global products of the Euclidean line with another Riemannian manifold. This is the situation in which
we choose to set ourselves to underline the connection between resolvent and Carleman estimates.

The setting is the product $N=\R \times M_{0}$ of the Euclidean (one dimensional) line and of an $(n-1)$-dimensional compact
Riemannian manifold  $(M_{0},g_{0})$ \textit{without boundary}, endowed with the product metric 
   $$ g=\dd x_{1}^2 + g_{0}. $$
We consider the Laplace-Beltrami operator on $N$
\begin{align*}
     P = \Delta_{g}=\d_{x_{1}}^2+\Delta_{g_{0}}.
\end{align*}
A natural limiting Carleman weight on this product manifold is $x_{1}$ and the corresponding conjugated operator reads
\begin{align*}
     \e^{\tau x_{1}}P\e^{-\tau x_{1}} = \d_{x_{1}}^2 - 2 \tau \d_{x_{1}} + \tau^2 + \Delta_{g_{0}}.
\end{align*}
\begin{thm}
\label{Carl:LpCarlThm}
    Let $N=\R \times M_{0}$ be the product of the Euclidean line and of an $(n-1)$-dimensional compact manifold $(M_0,g_0)$ endowed
    with the product metric $g=\dd x_{1}^2 + g_{0}$. For all compact intervals $I \subset \R$, there exist two constants $C,\tau_0>0$ 
    such that for all $\tau \in \R$ with $|\tau| \geq \tau_0$, we have
    \begin{align}
         \| \e^{\tau x_{1}} u\|_{L^{\frac{2n}{n-2}}(N)} \leq C \| \e^{\tau x_{1}} Pu\|_{L^{\frac{2n}{n+2}}(N)} 
    \end{align}
    when $u \in \mathcal{C}^{\infty}_0(N)$ is supported in $\mathring{I} \times M_0$.
\end{thm}

This Carleman estimate was obtained in \cite{DKSa}, and in \cite{Shen} when $M_0=\mathbf{T}^{n-1}$ is the $n-1$ dimensional torus, 
by a different method, which uses as well the spectral cluster estimates of Sogge but on $n-1$ dimensional manifolds.
By translating and scaling the estimate in the first variable, it is always possible to assume $I=[0,2\pi]$ without loss of generality,
which we do since this will lighten our notations.  The proof follows \cite{KRS} and is based on the construction of an inverse 
operator of the conjugated operator \eqref{Carl:ConjOp}; to ease such an inversion, inspired by
\cite{Hahner} (see also \cite[Section 3.2]{Salo}), we further conjugate the operator by a harmless oscillating factor
 \begin{align}
\label{Carl:ConjOp}
     -\e^{\tau x_{1}-\frac{i}{2}x_1}P\e^{-\tau x_{1}+\frac{i}{2}x_1} = \bigg(D_{x_{1}}+\frac{1}{2}\bigg)^2 + 2i \tau \bigg(D_{x_{1}}+\frac{1}{2}\bigg) 
     - \tau^2 - \Delta_{g_{0}}.
\end{align} 
We regard functions $u \in \D((0,2\pi) \times M_0)$ as smooth periodic functions in the first variable.
This corresponds to a change of perspective by replacing the manifold $N=\R \times M_0$ by the product 
$M=\mathbf{T} \times M_0$ of the one dimensional torus  $\mathbf{T}=\R/2\pi\Z$ and of the manifold $(M_0,g_0)$. 
We denote by $\lambda_{0}=0<\lambda_{1} \leq \lambda_{2} \leq \dots$ the sequence of
eigenvalues of $-\Delta_{g_{0}}$ on $M_{0}$ and $(\psi_{k})_{k \geq 0}$ the corresponding sequence of eigenfunctions
forming an orthonormal basis of $L^2(M_0)$
\begin{align*}
     -\Delta_{g_{0}}\psi_{k} = \lambda_{k} \psi_{k}.
\end{align*}
We denote by $\pi_{k} : L^2(M_{0}) \to L^2(M_{0})$ the projection on the linear space spanned by the eigenfunction $\psi_{k}$ so that 
\begin{align*}
     \sum_{k=0}^{\infty} \pi_{k} = {\rm Id}, \quad \sum_{k=0}^{\infty} \lambda_{k} \pi_{k} = -\Delta_{g_{0}}.
\end{align*}
The eigenvalues of the Laplacian $\Delta_{g}$ on $\mathbf{T} \times M_0$ are $-(j^2+\lambda_{k})$ 
with $j \in \Z$, $k \in \N$ and the corresponding eigenfunctions are%
\footnote{Here we are using in an essential manner the fact that the metric $g$ is a product, i.e. the fact that $x_{1}$ is a limiting Carleman weight.} 
    $$ \e^{i  j x_{1}} \psi_{k}. $$ 
We denote by $\pi_{j,k} : L^2(M) \to L^2(M)$ the projection on the linear space spanned by the eigenfunction $\e^{i j x_{1}} \psi_{k}$:
    $$ \pi_{j,k}f(x) = \frac{1}{2\pi}\bigg(\int_{0}^{2\pi} \e^{-i j y_{1}} \pi_{k}f(y_{1},x') \, \dd y_{1}\bigg) \, \e^{i j  x_{1}}, $$
and define the spectral clusters as
\begin{align*}
     \chi_{m} = \sum_{m \leq \sqrt{j^2+\lambda_{k}} < m+1} \pi_{j,k}, \quad m \in \N.
\end{align*}
Note that these are projectors $\chi_{m}^2=\chi_{m}$.
We end this paragraph by recalling the spectral cluster estimates of Sogge \cite{Sog,Sog1} that we will need:
\begin{align}
\label{Carl:spclusterEst}
\begin{split}
     \|\chi_{m}u\|_{L^{\frac{2n}{n-2}}(M)} &\leq C(1+m)^{\frac{1}{2}} \|u\|_{L^{2}(M)} \\ 
     \|\chi_{m}u\|_{L^{2}(M)} &\leq C(1+m)^{\frac{1}{2}} \|u\|_{L^{\frac{2n}{n+2}}(M)}.
\end{split}
\end{align}
The two inequalities are dual to each other and can be found in \cite[Corollary 5.1.2]{Sog}. 

We are now ready to reduce the proof of Carleman estimates to resolvent estimates. Recall that our goal is to prove 
\begin{align}
\label{Carl:ConjEst}
     \|u\|_{L^{\frac{2n}{n-2}}(M)} \leq C \|f\|_{L^{\frac{2n}{n+2}}(M)}
\end{align}
when $u \in \mathcal{C}^{\infty}_0((-\pi,\pi)\times M_0)$ and
\begin{align}
\label{Carl:ConjEq}
    \bigg(D_{x_{1}}+\frac{1}{2} \bigg)^2u+2i\tau \bigg(D_{x_{1}}+\frac{1}{2} \bigg)u-\Delta_{g_{0}}u-\tau^2 u = f.
\end{align}
The equation \eqref{Carl:ConjEq} is actually easy to solve: writing $f=\sum_{j,k} \pi_{j,k}f$ and similarly for $u$, the equation formally becomes
\begin{align*}
     \bigg(\bigg(j+\frac{1}{2}\bigg)^2 + 2i\tau \bigg(j+\frac{1}{2}\bigg) - \tau^2 + \lambda_k\bigg)\pi_{j,k}u=\pi_{j,k}f
\end{align*}
for $j \in \Z$ and for $k \in \N$. The symbol on the left is always nonzero provided that $|\tau| \geq 1$. Thus, the inverse operator may be written as  
\begin{equation*}
     G_{\tau} f =  \sum_{j=-\infty}^{\infty} \sum_{k=0}^{\infty} \frac{\pi_{j,k}f}{\big(j+\frac{1}{2}\big)^2+2i\big(j+\frac{1}{2}\big) \tau+\lambda_{k}-\tau^2}.
\end{equation*}

We now use Littlewood-Paley theory to localize in frequency with respect to the Euclidean variable $x_{1}$; we have
\begin{align*}  
    u = \sum_{\nu=0}^{\infty} u_{\nu}, \quad f = \sum_{\nu=0}^{\infty} f_{\nu}
\end{align*}
with
\begin{align*}
      u_{0}&=\frac{1}{2\pi}\bigg(\int_{0}^{2\pi}  u(y_{1},x') \, \dd y_{1}\bigg), \\
      u_{\nu}&=\sum_{2^{\nu-1} \leq |j| < 2^{\nu}} \frac{1}{2\pi} \bigg(\int_{0}^{2\pi} \e^{-i j  y_{1}} u(y_{1},x') \, \dd y_{1}\bigg) \, 
      \e^{i j x_{1}}, \quad \nu>0
 \end{align*}
and similarly for $f$. It suffices to prove
\begin{align}
\label{Carl:LocEst}
     \|u_{\nu}\|_{L^{\frac{2n}{n-2}}} \leq C \|f_{\nu}\|_{L^{\frac{2n}{n+2}}}
\end{align}
since then by summing the former inequalities and using Littlewood-Paley's theory
\begin{align*}
    \|u\|^2_{L^{\frac{2n}{n-2}}} \lesssim \sum_{\nu=0}^{\infty} \|u_{\nu}\|^2_{L^{\frac{2n}{n-2}}} 
    \lesssim \sum_{\nu=0}^{\infty} \|f_{\nu}\|^2_{L^{\frac{2n}{n-2}}}  \lesssim  \|f\|^2_{L^{\frac{2n}{n+2}}}
\end{align*}
and this is the  inequality \eqref{Carl:ConjEst} we were looking for.

We now concentrate on the microlocalized estimate \eqref{Carl:LocEst}. Because the conjugated operator \eqref{Carl:ConjOp} and the localization
in frequency $1_{[2^{\nu-1},2^{\nu})}(D_{x_1})$ commute%
\footnote{Here again we are using the fact that the weight $x_{1}$ is a limiting Carleman weight, i.e. that the metric $g_{0}$ does not depend on $x_{1}$.},
we have
    $$ u_{\nu} = G_{\tau}f_{\nu}. $$
We denote
\begin{align*}
      R(z) = \bigg(\bigg(D_{x_{1}}+\frac{1}{2}\bigg)^2 - \Delta_{g_{0}} + z\bigg)^{-1}
\end{align*}    
the resolvent associated with the elliptic operator $\big(D_{x_{1}}+\frac{1}{2}\big)^2 - \Delta_{g_{0}}$. In the case $\nu=0$, we have 
\begin{align}
\label{Carl:lowfrequency}
      u_0 = G_{\tau}f_0 = R(-\tau^2+i\tau) f_0
\end{align}
so \eqref{Carl:LocEst} is a resolvent estimate with $z=-\tau^2+i\tau$. When $\nu \geq 1$, we want to evaluate on microlocalized
functions the error made by freezing the operator $2 \tau D_{x_{1}}$ into $2^{\nu}\tau$ in \eqref{Carl:ConjOp}. This amounts to replacing $G_{\tau}$ 
by the resolvent $R(-\tau^2+i(2^{\nu}+1)\tau)$. More precisely, we want to estimate the term
\begin{align}
     \big(R(-\tau^2+i (2^{\nu}+1)\tau) - G_{\tau}\big)f_{\nu} .
\end{align}
This can be explicitly computed           
\begin{align*}
     \big(R(-\tau^2+i(2^{\nu}+1)\tau) -G_{\tau}\big)  f_{\nu} =   
     \sum_{j=-\infty}^{\infty} \sum_{k=0}^{\infty} a_{jk}^{\nu}(\tau) \,  \pi_{j,k}f_{\nu} 
\end{align*}
where
\begin{multline*}
    a_{jk}^{\nu}(\tau)  =\\   \frac{i\tau  (2^{\nu}-2j)1_{[2^{\nu-1},2^{\nu})}(j)}
    {\Big(\big(j+\frac{1}{2}\big)^2+2i\tau \big(j+\frac{1}{2}\big) -\tau^2 + \lambda_{k}\Big)
    \Big(\big(j+\frac{1}{2}\big)^2+i(2^{\nu}+1)\tau -\tau^2 + \lambda_{k}\Big)}.
\end{multline*}

Using the spectral cluster estimates \eqref{Carl:spclusterEst} and the fact that spectral clusters are projectors, we have the following string of estimates
\begin{align*}
     \big\|\big(R(-\tau^2+i(2^{\nu}+1)\tau) - &G_{\tau}\big)f_{\nu}\big\|_{L^{\frac{2n}{n-2}}(M)} \\
     & \lesssim \sum_{m=0}^{\infty} (1+m)^{\frac{1}{2}} \big\|\chi_m\big(R(-\tau^2+i(2^{\nu}+1)\tau) - R_{\tau}\big) f_{\nu}\big\|_{L^{2}(M)} \\
     & \leq \sum_{m=0}^{\infty} (1+m)^{\frac{1}{2}} \sup_{m \leq \sqrt{j^2+\lambda_{k}} < m+1} \big|a^{\nu}_{jk}(\tau)\big| 
     \times  \|\chi_{m} f_{\nu}\|_{L^2(M)} 
\end{align*}
and using again \eqref{Carl:spclusterEst}, we get 
\begin{multline*}
     \big\|\big(R(-\tau^2+i(2^{\nu}+1)\tau) - G_{\tau}\big)f_{\nu}\big\|_{L^{\frac{2n}{n-2}}(M)}  \\
     \lesssim \bigg(\sum_{m=0}^{\infty} (1+m)  \sup_{m \leq \sqrt{j^2+\lambda_{k}} < m+1} \big|a^{\nu}_{jk}(\tau)\big| \bigg) 
     \times \|f_{\nu}\|_{L^{\frac{2n}{n+2}}(M)}.
\end{multline*}
We now prove that the above series converges and is uniformly bounded with respect to $\tau$ and $\nu$; we have
    $$ \sup_{m \leq \sqrt{j^2+\lambda_{k}} < m+1} \big|a^{\nu}_{jk}(\tau)\big| \lesssim  
         \frac{2^{\nu}|\tau|}{(m^2-\tau^2)^2+ 4^{\nu+1} \tau^2} $$
as well as
\begin{align*}
     \sum_{m=0}^{\infty} \frac{2^{\nu}|\tau| (1+m)}{(m^2-\tau^2)^2+ 4^{\nu+1} \tau^2} \lesssim
     \int_{0}^{\infty}  \frac{2^{\nu} |\tau| t}{(t^2-\tau^2)^2+ 4^{\nu+1} \tau^2} \, \dd t
\end{align*}
and if we perform the change of variables $s=4^{-\nu-1}\tau^{-2}(t^2-\tau^2)$ in the right-hand side integral, we obtain the bound
    $$  \sum_{m=0}^{\infty} \frac{2^{\nu}|\tau| (1+m)}{(m^2-\tau^2)^2+ 4^{\nu+1} \tau^2} \lesssim \int_{-\infty}^{\infty} \frac{\dd s}{s^2+1}. $$
Summing up our computations, we have the error estimate
\begin{align*}
     \big\|\big(R(-\tau^2+i(2^{\nu}+1)\tau) - G_{\tau}\big)f_{\nu}\big\|_{L^{\frac{2n}{n-2}}(M)}  
     \lesssim \|f_{\nu}\|_{L^{\frac{2n}{n+2}}(M)},
\end{align*}
this means that in order to obtain \eqref{Carl:LocEst}, it is enough to prove the resolvent estimate
\begin{align*}
     \big\|R(-\tau^2+i(2^{\nu}+1)\tau)f_{\nu}\big\|_{L^{\frac{2n}{n-2}}(M)}  
     \lesssim \|f_{\nu}\|_{L^{\frac{2n}{n+2}}(M)}
\end{align*}
with a constant which is uniform in $\tau$ and $\nu$.  

The conclusion of these computations is that Carleman estimates reduce to resolvent estimates of the form
\begin{align}
\label{Carl:ResolConj}
     \|u\|_{L^{\frac{2n}{n-2}}(M)} \lesssim  \bigg\| \bigg(D_1+\frac{1}{2}\bigg)^2-\Delta_{g_0}+z\bigg)u\bigg\|_{L^{\frac{2n}{n+2}}(M)}
\end{align}
with
\begin{align*}
     z = -\tau^2+i \varrho\tau \quad (\varrho \geq 1).
\end{align*}
After changing $u$ into $\e^{-\frac{i}{2}x_1}u$ in \eqref{Carl:ResolConj}, the resolvent estimate  is a consequence of Theorem \ref{Intro:MainThm}
since
          $$ \re \sqrt{z} = \sqrt{\frac{\re z +|z|}{2}} = \tau \sqrt{\frac{\sqrt{\tau^2+\varrho^2}-1}{2}} \geq 1 $$
provided $|\tau|$ is large. 
\end{section}
%
%

%
%
\end{document}